\documentclass[11pt]{amsproc}

\usepackage{mathtext}
\usepackage[cp1251]{inputenc}
\usepackage[T2A]{fontenc}
\RequirePackage{srcltx}

\usepackage[dvips]{graphicx}
\usepackage{amsmath}
\usepackage{amssymb}
\usepackage{amsxtra}
\usepackage{latexsym}
\usepackage{ifthen}
\usepackage{mathrsfs}

\def\N{{{\Bbb N}}}
\def\Z{{{\Bbb Z}}}
\def\T{{{\Bbb T}}}
\def\R{{\Bbb R}}

\def\a{{\alpha }}
\def\D{{\Delta }}

\def\a{{\alpha}}

\def\d{{\delta}}
\def\e{{\varepsilon}}
\def\s{{\sigma}}
\def\vp{{\varphi}}

\def\g{{\gamma }}
\def\w{{\omega }}

\def\E{\mathcal{E}}
\def\Sp{\mathcal{S}}

\def\Dl{\bar{{\Delta}}}

\def\){\right)}
\def\({\left(}

\numberwithin{equation}{section}

\newtheorem{corollary}{Corollary}[section]
\newtheorem{lemma}{Lemma}[section]
\newtheorem{theorem}{Theorem}[section]
\newtheorem{proposition}{Proposition}[section]

\def\R{\Bbb R}

\def\XXint#1#2#3{{\setbox0=\hbox{$#1{#2#3}{\int}$}
     \vcenter{\hbox{$#2#3$}}\kern-.5\wd0}}

\par

\bigskip
\bigskip

\begin{document}

\title[Best approximations and moduli of smoothness of functions]{Best approximations and moduli of smoothness of functions and their derivatives in $L_p$, $0<p<1$}

\author[Yurii
Kolomoitsev]{Yurii
Kolomoitsev$^{\text{a}, \text{b}, \text{1}}$}

%
%
%
\thanks{$^\text{a}$Universit\"at zu L\"ubeck,
Institut f\"ur Mathematik,
Ratzeburger Allee 160,
23562 L\"ubeck}

\thanks{$^b$Institute of Applied Mathematics and Mechanics of NAS of Ukraine,
Dobrovol's'kogo str.~1, Slov’yans’k, Donetsk region, Ukraine, 84100}

\thanks{$^1$Supported by the project AFFMA that has received funding from the European Union's Horizon 2020 research and innovation
programme under the Marie Sklodowska-Curie grant agreement No 704030.}

\thanks{E-mail address: kolomoitsev@math.uni-luebeck.de, kolomus1@mail.ru}

\date{\today}
\subjclass[2010]{41A10, 41A15, 41A17, 41A25, 41A28} \keywords{spaces $L_p$, $0<p<1$, moduli of smoothness, best approximation, simultaneous approximation, trigonometric and algebraic polynomials,  splines}

\begin{abstract}
Several new inequalities for moduli of smoothness and errors of the best approximation of a function and its derivatives in the spaces $L_p$, $0<p<1$, are obtained.
For example, it is shown that for any $0<p<1$ and $k,\,r\in \N$
$$
    \w_{r+k}(f,\d)_p\leq C({p,k,r})\d^{r+\frac{1}{p}-1}\(\int_0^\d\frac{\w_{k}(f^{(r)},t)_p^p}{t^{2-p}}{\rm d}t\)^\frac{1}{p},
$$
where the function $f$ is such that $f^{(r-1)}$ is absolutely continuous. Similar inequalities are obtained for the Ditzian-Totik moduli of smoothness and the error of the best approximation of functions by trigonometric and algebraic polynomials and splines. As an application, positive results about simultaneous approximation of a function and its derivatives by the mentioned approximation methods in the spaces $L_p$, $0<p<1$, are derived.
\end{abstract}

\maketitle

\section{Introduction}
\label{intro}

Let $A$ be a finite interval $[a,b]$ or the unit circle $\T\cong [0,2\pi)$. As usual, $L_p=L_p(A)$, $0<p<\infty$, denotes the space of all measurable function $f$ on $A$ such that
$$
\Vert f \Vert_p=\Vert f \Vert_{L_p(A)}=\(\int_{A}|f(x)|^p{\rm d}x\)^{\frac{1}{p}}<\infty
$$
and $W_p^r(A)$, $1\leq p\leq\infty$, $r\in \N$, denotes the Sobolev space of functions, that is  $f\in W_p^r(A)$ if $f^{(r-1)}\in { AC}(A)$ (absolutely continuous functions on $A$) and $f^{(r)}\in L_p(A)$.

Measuring the smoothness of a function by differentiability is too crude for many purposes of analysis. Subtler measurements are provided by moduli of smoothness. Recall that for $f\in L_p$, the classical (non-periodic and periodic) modulus of smoothness of order $r\in \mathbb{N}$ is defined by
$$
\w_r(f,\d)_p=\w_r(f,\d)_{L_p(A)}=\sup_{0<h\leq \d}\Vert \D_h^rf\Vert_{L_p(A_{rh})}\,,
$$
where
$$
\D_h^rf(x)=\sum_{\nu=0}^r\binom{r}{\nu}(-1)^\nu f(x+\nu h),
$$
$ \binom{r}{\nu}=\frac{r (r-1)\dots
(r-\nu+1)}{\nu!}$, $\binom{r}{0}=1$, and $A_{rh}=[a,b-rh]$ in the case $A=[a,b]$ or $A_{rh}=\mathbb{T}$ in the case $A=\mathbb{T}$.
We also use the notation $\w_0(f,\d)_p=\Vert f\Vert_p$.

It is well-known (see \cite{DeLo}, p.~46) that for any function $f\in W_p^r(A)$, $1\leq p<\infty$, and $k,r\in \Z_+$
\begin{equation}\label{eq.Sec1.1}
    \w_{r+k}(f,\d)_p\leq \d^r\w_{k}(f^{(r)},\d)_p.
\end{equation}
It is also possible   to estimate $\w_{k}(f^{(r)},\d)_p$ from above by $\w_{r+k}(f,\d)_p$. Such estimate is given by
the following weak-type inverse inequality to~\eqref{eq.Sec1.1}: for $f\in L_p$, $1\leq p<\infty$, and $k,r\in \N$ one has
\begin{equation}\label{eq.Sec1.2}
    \w_{k}(f^{(r)},\d)_p\leq C_{r}\int_0^\d\frac{\w_{r+k}(f,t)_p}{t^{r+1}}{\rm d}t
\end{equation}
(see Johnen and Scherer~\cite{johnen}, see also \cite[p.~178]{DeLo}).
Inequalities \eqref{eq.Sec1.1} and \eqref{eq.Sec1.2} have important applications in theory of functions and approximation theory and have been intensively studied in different settings in the case of Banach spaces (see, e.g.,~\cite[Ch.~4]{BeSh}, \cite[Ch.~2 and Ch.~6]{DeLo}, and~\cite{simonov-sb}).

In contrast, in the spaces $L_p$, $0<p<1$, there are only some partial  results related to (weak) inverse inequalities  and some examples of functions for which the classical direct inequalities of type~\eqref{eq.Sec1.1} are impossible.
Thus, Ditzian and Tikhonov~\cite{DiTi07} proved that for any periodic function $f\in L_p(\T)$, $0<p<1$, and $k,r\in \N$ one has
\begin{equation}\label{eq.Sec1.2+}
    \w_{k}(f^{(r)},\d)_{L_p(\T)}\le C_{p,k,r}\(\int_0^\d
    \frac{\w_{r+k}(f,t)_{L_p(\T)}^p}{t^{pr+1}}{\rm d}t\)^\frac1p.
\end{equation}
At the same time, it is known that inequality (\ref{eq.Sec1.1}) is no longer valid for a general $f$ in the case $0<p<1$, even if we assume that $f\in C^\infty$ (see \cite{Kop06}). Moreover, in the monograph of Petrushev and  Popov~\cite[p.~188]{PePo}, it was mentioned that
"\emph{there is no upper estimate of $\w_{k}(f,\d)_p$ by $\w_{k-1}(f^\prime,\d)_p$ in the case $0<p<1$}".
Surprisingly, it turns out that such estimation is possible  but in terms of weak-type inequalities related to~\eqref{eq.Sec1.2} and~\eqref{eq.Sec1.2+}.
Namely, in this paper, we show that for any $0<p<1$, $k,\,r\in \N$, and any function $f$ such that $f^{(r-1)}\in AC$ we have the following analogue of~\eqref{eq.Sec1.1}
\begin{equation*}
    \w_{r+k}(f,\d)_p\leq C_{p,k,r}\d^{r+\frac{1}{p}-1}\(\int_0^\d\frac{\w_{k}(f^{(r)},t)_p^p}{t^{2-p}}{\rm d}t\)^\frac{1}{p}
\end{equation*}
(see Theorem~\ref{th1T}).

%
%
%


A similar situation arises in studying inequalities for the error of polynomial approximation. Let us consider, for example, the case of approximation of functions by trigonometric polynomials. Recall that the error of the best trigonometric polynomial approximation is given by
$E_n(f)_p=\inf_{T\in \mathcal{T}_n}\Vert f-T\Vert_{L_p(\T)}$, where $\mathcal{T}_n$  denotes the set of all trigonometric polynomials of degree at most~$n$.

It is well-known (see~\cite[p. 206]{DeLo}) that for any function $f\in W_p^r(\T)$, $1\leq p<\infty$, and $r\in \N$ one has
\begin{equation}\label{eqP1}
  E_n(f)_p\le C_{r}n^{-r}E_{n}(f^{(r)})_p.
\end{equation}
In the case $0<p<1$,  inequality~\eqref{eqP1} does not hold. In particular, from the result of Kopotun~\cite{Kop95} (see also Ivanov~\cite{Iv}) it follows that
for every $C>0$, $B\in \R$, $0<p<1$, and $n\in \N$ there exists a function $f\in AC(\T)$ such that
\begin{equation}\label{eqKop}
  E_n(f)_p>Cn^B \Vert f'\Vert_{L_p(\T)}.
\end{equation}
We show that for any $0<p<1$ and a function $f$ such that $f^{(r-1)}\in AC(\T)$, $r\in \N$, the following counterpart of~\eqref{eqP1}
\begin{equation*}
  E_n(f)_p\leq{C_{p,r}}{n^{-r}}\(E_n(f^{(r)})_p+n^{1-\frac{1}{p}}\(\sum\limits_{\nu=n+1}^\infty\nu^{-p}E_\nu(f^{(r)})_p^p\)^{\frac{1}{p}}\)
\end{equation*}
is true (see Theorem~\ref{thsec3.1}).

A close problem to the mentioned above inequalities is the problem of studying simultaneous approximation of functions and their derivatives in $L_p$. Let us recall the classical result of Czipzer and Freud \cite{CzFr} about simultaneous approximation of periodic functions by trigonometric polynomials: if $f\in W_p^r(\T)$, $1\leq p<\infty$, and $r\in\N$, then
\begin{equation}\label{simul}
  \|f^{(r)}-T_n^{(r)}\|_{L_p(\T)}\leq C_{r} E_n(f^{(r)})_{p}\,,
\end{equation}
where the polynomials $T_n\in\mathcal{T}_n$ are such that $\|f-T_n\|_p= E_n(f)_p$.
In this paper, we prove that, in the case $0<p<1$, an analogue of inequality~\eqref{simul} has the following form
\begin{equation}\label{eqSS1}
  \|f^{(r)}-T_n^{(r)}\|_{L_p(\T)}\leq C_{p,r}\(E_n(f^{(r)})_p+n^{1-\frac{1}{p}}\(\sum\limits_{\nu=n+1}^\infty \nu^{-p}E_\nu(f^{(r)})_p^p\)^{\frac{1}{p}}\)
\end{equation}
(see Theorem~\ref{thsec3.2}).


It is worth mentioning some results about simultaneous approximation of function and its derivatives by algebraic polynomials. Let $\mathcal{P}_n$ denote the set of all algebraic polynomials of degree at most $n$. Kopotun~\cite{Kop95_2}  proved that for any function $f$ such that $f^{(r-1)}\in AC[-1,1]$ and $f^{(r)}\in L_p[-1,1]$, $1\le p<\infty$, and  $s\in \N$, there exists an algebraic polynomial $P_n\in \mathcal{P}_n$ such that
\begin{equation*}
  \Vert f^{(k)}-P_n^{(k)}\Vert_{L_p[-1,1]}\le C_{s,r}\w_{s+r-k}^\vp (f^{(k)},n^{-1})_p,\quad 0\le k\le r,
\end{equation*}
where $\w_s^\vp(g,\d)_p$ is the Ditzian-Totik modulus of smoothness of order $s$ in $L_p[-1,1]$.
At the same time, Ditzian~\cite{Di} showed that \emph{"for $0<p<1$  simultaneous polynomial approximation is not possible"}. More precisely, Ditzian proved that there exists a function $f\in AC[0,1]$ such that for any $0<p<1$ and $P_n\in \mathcal{P}_n$ the following inequalities
\begin{equation}\label{eqD1}
  \Vert f-P_n\Vert_{L_p[-1,1]}\le C\w_2(f,n^{-1})_p
\end{equation}
and
\begin{equation}\label{eqD2}
\Vert f'-P_n'\Vert_{L_p[-1,1]}\le C\w_1(f',n^{-1})_p
\end{equation}
cannot hold simultaneously with a constant $C$ independent of $f$ and $n$.

Kopotun~\cite{Kop98} improved this result by showing that if $f$ is assumed to be $k$-monotone function, then simultaneous approximation of $f$ and its derivatives is possible for $p<1$. In particular, if $f$ is a convex function, then there exists $P_n\in \mathcal{P}_n$ such that~\eqref{eqD1} and~\eqref{eqD2} hold simultaneously with the constant $C=C(p)$.
In this paper, based on inequality~\eqref{eqSS1}, we obtain another improvement of the above Ditzian's result (see Section~4). Moreover, we derive several results about simultaneous approximation of a function and its derivatives by splines in $L_p$, $0<p<1$.

Let us mention that in the recent papers~\cite{KLP} and~\cite{KP} it was studied similar problems
concerning approximation of functions by trigonometric and algebraic polynomials in the H\"older spaces $H_p^\a$ with $0<p<1$.


The paper is organized as follows: In Section~2, we consider periodic functions in the spaces $L_p(\T)$, $0<p<1$. In particular, in Subsection~2.1, we study inequalities for the errors of the best approximation of functions and their derivatives; in Subsection~2.2, we obtain new inequalities for moduli of smoothness of functions and their derivatives; in Subsection~2.3, we show the sharpness of the main results of the paper in the periodic case. In Section~3, we derive analogues of the main results from Section~2 in the case of the non-periodic moduli of smoothness and approximation of functions by splines in the space $L_p[0,1]$, $0<p<1$. In Section~4, the previous problems are considered within the framework  of the Ditzian-Totik moduli of smoothness and the  approximation of functions by algebraic polynomials in the spaces $L_p[-1,1]$, $0<p<1$.

In what follows, we denote by $C$ some  positive constants depending on the indicated parameters.

\section{Approximation of functions by trigonometric polynomials. Periodic moduli of smoothness}

In this section, we take $A=\T$ and denote $\Vert \cdot\Vert_p=\Vert \cdot\Vert_{L_p(\T)}$.
Let $\mathcal{T}_n$ be the set of all trigonometric
polynomials of order at most $n$
and let
$$
E_n(f)_p=\inf_{T\in \mathcal{T}_n}\Vert f-T\Vert_p
$$
be the error of the best approximation of a function $f$ by trigonometric polynomials of order at most $n$ in $L_p(\T)$.
A trigonometric polynomial $T_n\in \mathcal{T}_n$  is called a polynomial of the best approximation of $f$ in $L_p(\T)$ if
$$
\Vert f-T_n\Vert_p=E_n(f)_p.
$$

\subsection{\textbf{Inequalities for the error of the best approximations of functions by trigonometric polynomials}}

In this section, one of the main results is the following counterpart of inequality~\eqref{eqP1}
in the case  $0<p<1$.

\begin{theorem}\label{thsec3.1}
{\it Let $0<p<1$, $r\in \N$, and let $f$ be such that $f^{(r-1)}\in AC(\T)$ and
\begin{equation}\label{eqthsec3.1.1}
  \sum\limits_{\nu=1}^\infty \nu^{-p}E_\nu(f^{(r)})_p^p<\infty\,.
\end{equation}
Then for any $n\in \N$ we have
\begin{equation}\label{eqthsec3.1.2}
  E_n(f)_p\leq\frac{C}{n^r}\(E_n(f^{(r)})_p+n^{1-\frac{1}{p}}\(\sum\limits_{\nu=n+1}^\infty\nu^{-p}E_\nu(f^{(r)})_p^p\)^{\frac{1}{p}}\)\,,
\end{equation}
where $C$ is a constant independent of $f$ and $n$.}
\end{theorem}

The proof of this theorem is based on the next three important results in the theory of approximation.
The first one is the Jackson-type theorem in $L_p(\T)$, $0<p<1$ (see~\cite{SO} and also~\cite{SKO} and~\cite{Iv}).
\begin{lemma}\label{lemJT}
Let $f\in L_p(\T)$, $0<p<1$, $k\in \N$, and $n\in\N$. Then
\begin{equation*}
  E_n(f)_p\le C\w_k\(f,\frac 1n\)_p,
\end{equation*}
where $C$ is a constant independent of $f$ and $n$.
\end{lemma}

The second result is the  Stechkin-Nikolskii type inequality (see~\cite{DHI}).

\begin{lemma}\label{lemNST}
Let $0<p<\infty$, $n\in\N$, $0<h\le \pi/n$, and $r\in\N$. Then for any  $T_n \in \mathcal{T}_n$, we have
\begin{equation*}
  h^r\Vert T_n^{(r)}\Vert_p\asymp \Vert \D_h^r T_n\Vert_p,
\end{equation*}
where $\asymp$ is a two-sided inequality with positive constants independent of $T_n$ and $h$.
\end{lemma}

The third result is the well-known Nikolskii inequality of different metrics (see, e.g.,~\cite[Ch. 4, \S2]{DeLo}).
\begin{lemma}\label{lemNikT} {\it Let $0<p<q<\infty$. Then for any  $T_n\in \mathcal{T}_n$, $n\in \N$, one has
\begin{equation*}
  \Vert T_n \Vert_q\le C n^{\frac1p-\frac1q}\Vert T_n\Vert_p,
\end{equation*}
where  the constant $C$ depends only on $p$ and $q$.}
\end{lemma}

We need the following properties of
moduli of smoothness
 (see \cite[Ch. 2, \S~7 and  Ch. 12, \S~5]{DeLo} or
\cite[Ch. 4]{TB}). {\it Let $A=\T$ or $A=[0,1]$ and let $f,g\in L_p(A)$, $0<p< \infty$, $r\le k$,
$k,r\in\N$. Then
\begin{equation}\label{eqM0}
       \omega_k(f+g,\d)_p^{p_1}\le \omega_k(f,\d)_p^{p_1}+\omega_k(g,\d)_p^{p_1},\quad \d>0,
\end{equation}
\begin{equation}\label{eqM1}
       \omega_k(f,\d)_p\le 2^{\frac{k-r}{p_1}}\omega_r(f,\d)_p\le 2^\frac k{p_1}\Vert f\Vert_{L_p(A)},\quad \d>0,
\end{equation}
\begin{equation}\label{eqM2}
       \omega_r(f,\lambda \d)_p\le r^{\frac 1{p_1}-1}(1+\lambda)^{\frac
       1{p_1}+r-1}\omega_r(f,\d)_p,\quad \d>0,\quad \lambda>0\,,
\end{equation}
where $p_1=\min(p,\,1)$.}

\begin{proof}\emph{of\,\, Theorem~\ref{thsec3.1}\,\,\,}
Let $U_n\in\mathcal{T}_n$, $n\in\N$, be such that
$$
\|f^{(r)}-U_n\|_p=E_n(f^{(r)})_p
$$
and let $T_n\in \mathcal{T}_n$, $n\in\N$, be such that
$$
T_n^{(r)}(x)=U_n(x)-\frac{1}{2\pi}\int\limits_0^{2\pi} U_n(x){\rm d}x.
$$

Choosing $m\in\N$ such that $2^{m-1}\leq n<2^m$, we have
\begin{equation}\label{eqthsec3.1.3}
  E_n(f)_p^p\leq E_n(T_{2^m})_p^p+E_n(f-T_{2^m})_p^p\,.
\end{equation}
Let us estimate $E_n(T_{2^m})_p$. Denoting
$$
\tau_u(x)=\tau_{u,2^m,n}(x)=\D_u^1(T_{2^m}(x)-T_n(x)),\quad u>0\,,
$$
and applying Lemma~\ref{lemJT}, we obtain
\begin{equation}\label{eqthsec3.1.4}
\begin{split}
E_n(T_{2^m})_p&=E_n(T_{2^m}-T_n)_p\leq C\w_{r+1}(T_{2^m}-T_n,\,n^{-1})_p\\
&=C \sup\limits_{0<h\leq n^{-1}}\|\D_h^r \tau_h\|_p\leq C \sup\limits_{0<h\leq n^{-1}} \sup\limits_{u>0}\|\D_h^r \tau_u\|_p\\&\leq C \sup\limits_{u>0}\w_r(\tau_u,\,n^{-1})_p\,.
\end{split}
\end{equation}
Next, taking into account that $\tau_u\in \mathcal{T}_{2^m}$ for any fixed $u>0$ and applying inequalities (\ref{eqM1}), (\ref{eqM2}), and Lemma~\ref{lemNST}, we get
\begin{equation}\label{eqthsec3.1.5}
\begin{split}
\w_r(\tau_u,\,n^{-1})_p&\leq C\w_r(\tau_u,\,2^{-m})_p\leq C 2^{-mr}\|\tau_u^{(r)}\|_p\\
&=C 2^{-mr}\|\D_u^1(T_{2^m}^{(r)}-T_n^{(r)})\|_p=C 2^{-mr}\|\D_u^1(U_{2^m}-U_n)\|_p\\
&\leq C n^{-r}\|U_{2^m}-U_n\|_p\leq C n^{-r} E_n(f^{(r)})_p\,.
\end{split}
\end{equation}
Combining (\ref{eqthsec3.1.4}) and (\ref{eqthsec3.1.5}), we derive
\begin{equation}\label{eqthsec3.1.6}
E_n(T_{2^m})_p\leq C n^{-r} E_n(f^{(r)})_p\,.
\end{equation}

Now let us consider the second term in the right-hand side of (\ref{eqthsec3.1.3}).
First, we show
\begin{equation}\label{eqthsec3.1.7}
E_n(f-T_{2^m})_p^p\leq \sum\limits_{\mu=m}^\infty E_n(T_{2^{\mu+1}}-T_{2^\mu})_p^p\,.
\end{equation}
It is easy to see that for any $N>m$ we have
\begin{equation*}
E_n(f-T_{2^m})_p^p\leq \sum\limits_{\mu=m}^{N-1} E_n(T_{2^{\mu+1}}-T_{2^\mu})_p^p+E_n(f-T_{2^N})_p^p\,.
\end{equation*}
Thus, to show (\ref{eqthsec3.1.7}), it is enough to verify that
\begin{equation}\label{eqthsec3.1.9}
E_n(f-T_{2^N})_p^p\rightarrow0 \quad \textrm{as} \quad N\rightarrow\infty\,.
\end{equation}
Indeed, by Lemma~\ref{lemNikT}, we have
\begin{equation*}
\begin{split}
\sum\limits_{\mu=1}^\infty\|U_{2^{\mu+1}}-U_{2^\mu}\|_1^p&\leq C\sum\limits_{\mu=1}^\infty 2^{(1-p)\mu}\|U_{2^{\mu+1}}-U_{2^\mu}\|_p^p\\
&\leq C\sum\limits_{\mu=1}^\infty 2^{(1-p)\mu} E_{2^\mu}(f^{(r)})_p^p\leq C\sum\limits_{\nu=1}^\infty \nu^{-p} E_{\nu}(f^{(r)})_p^p\,.
\end{split}
\end{equation*}
In view of~\eqref{eqthsec3.1.1}, this implies that there exists $g\in L_1(\T)$ such that
$U_{2^\mu}\rightarrow g$ as $\mu\rightarrow\infty$ in $L_1(\T)$. By the definition of $U_n$, we know that $U_{2^\mu}\rightarrow f^{(r)}$ as
$\mu\rightarrow\infty$ in $L_p(\T)$. Therefore, $g=f^{(r)}$ a.e. on $\T$ and
\begin{equation}\label{zvezda}
  U_{2^\mu}\rightarrow f^{(r)}\quad \text{as}\quad \mu\rightarrow\infty\quad \text{in}\quad L_1(\T).
\end{equation}
At the same time, using H\"older's inequality and the estimate $E_n(f)_1\leq Cn^{-r}E_n(f^{(r)})_1$ (see \cite{DeLo}, p.~206), we obtain
\begin{equation}\label{ZZZ}
\begin{split}
E_n\(f-T_{2^N}\)_p&\leq C E_n\(f-T_{2^N}\)_1\leq C n^{-r} E_n(f^{(r)}-T_{2^N}^{(r)})_1\\
&= C n^{-r} E_n(f^{(r)}-U_{2^N})_1\leq C n^{-r} \|f^{(r)}-U_{2^N}\|_1\,.
\end{split}
\end{equation}
By~\eqref{ZZZ} and~\eqref{zvezda},  we get (\ref{eqthsec3.1.9}) and, hence, (\ref{eqthsec3.1.7}).

Now, using Lemma~\ref{lemJT}, inequalities (\ref{eqM2}) and (\ref{eqM1}), and applying the same arguments as in \eqref{eqthsec3.1.4} and (\ref{eqthsec3.1.5}) to the function
$\tau_u(x)=\tau_{u,2^{\mu+1},2^\mu}(x)=\D_u^1(T_{2^{\mu+1}}(x)-T_{2^\mu}(x))$, we derive
\begin{equation}\label{eqthsec3.1.10}
\begin{split}
E_n(T_{2^{\mu+1}}-T_{2^\mu})_p& \leq C\w_{r+1}(T_{2^{\mu+1}}-T_{2^\mu},\,n^{-1})_p\leq C \sup\limits_{u>0}\w_r(\tau_u,\,n^{-1})_p\\
&\leq C(2^{\mu+1} n^{-1})^{r+\frac{1}{p}-1}\sup\limits_{u>0}\w_r(\tau_u,\,2^{-\mu-1})_p\\
&\leq C n^{-r-\frac{1}{p}+1}2^{\mu\(\frac{1}{p}-1\)}\sup\limits_{u>0}\|\tau_u^{(r)}\|_p\\
&\leq C n^{-r-\frac{1}{p}+1}2^{\mu\(\frac{1}{p}-1\)}\sup\limits_{u>0}\Vert \D_u^1 (T_{2^{\mu+1}}^{(r)}-T_{2^{\mu}}^{(r)})\Vert_p\\
&\leq C n^{-r-\frac{1}{p}+1}2^{\mu\(\frac{1}{p}-1\)}\sup\limits_{u>0}\Vert \D_u^1 (U_{2^{\mu+1}}-U_{2^{\mu}})\Vert_p\\
&\leq C n^{-r-\frac{1}{p}+1}2^{\mu\(\frac{1}{p}-1\)}\sup\limits_{u>0}\Vert U_{2^{\mu+1}}-U_{2^{\mu}}\Vert_p\\
&\leq C n^{-r-\frac{1}{p}+1}2^{\mu\(\frac{1}{p}-1\)} E_{2^\mu}(f^{(r)})_p\,.
\end{split}
\end{equation}
In the third inequality, we take into account that $n<2^m\le 2^{\mu+1}$.

Thus, combining (\ref{eqthsec3.1.7}) and (\ref{eqthsec3.1.10}), we obtain
\begin{equation}\label{eqthsec3.1.11}
\begin{split}
E_n(f-T_{2^m})_p^p& \leq Cn^{-rp-1+p}\sum\limits_{\mu=m}^\infty 2^{(1-p)\mu} E_{2^\mu}(f^{(r)})_p^p\\
&\leq Cn^{-rp-1+p}\sum\limits_{\nu=n+1}^\infty \nu^{-p} E_{\nu}(f^{(r)})_p^p\,.
\end{split}
\end{equation}
Finally, combining (\ref{eqthsec3.1.3}), (\ref{eqthsec3.1.6}), and (\ref{eqthsec3.1.11}), we get (\ref{eqthsec3.1.2}).

The theorem is proved.
\end{proof}

Now let us consider an inverse inequality to (\ref{eqthsec3.1.2}).
For this, we need the notion of the
derivative in the sense of $L_p(\T)$ spaces. We say that a
function $f\in L_p(\mathbb{T})$, $0<p<\infty$, has the derivative
of order $k\in\N$ in the sense of $L_p(\T)$ if there exists a function $g$ such that
\begin{equation}\label{senseDerT}
  \bigg\Vert \frac{\D_h^k f}{h^k}-g\bigg\Vert_{p}\to0 \quad \textrm{as} \quad h\rightarrow 0\,.
\end{equation}
In this case, we write $g=f^{(k)}$.

Ivanov \cite{Iv} proved the following result.

\begin{theorem}\label{thsec3.A}
Let $f\in L_p(\mathbb{T})$, $0<p<1$, and let for some $k\in \N$
\begin{equation*}
\sum\limits_{\nu=1}^\infty \nu^{kp-1}E_\nu(f)_p^p<\infty\,.
\end{equation*}
Then $f$ has the derivative $f^{(k)}$ in the sense of $L_p(\T)$ and for any $n\in\N$
\begin{equation}\label{eqthsec3.A.2}
\|f^{(k)}-T_n^{(k)}\|_p\leq C\(n^k E_n(f)_p+\(\sum\limits_{\nu=n+1}^\infty \nu^{kp-1}E_\nu(f)_p^p\)^{\frac{1}{p}}\)\,,
\end{equation}
where $T_n\in\mathcal{T}_n$, $n\in\N$, are such that $\|f-T_n\|_p=E_n(f)$ and $C$ is a constant independent of $f$ and $n$.
\end{theorem}

Combining Theorem~\ref{thsec3.1} and Theorem~\ref{thsec3.A}, we obtain the following theorem about simultaneous approximation in the spaces $L_p(\mathbb{T})$.

\begin{theorem}\label{thsec3.2}
{\it Let $0<p<1$, $r\in \N$, and let $f$ be such that $f^{(r-1)}\in AC(\T)$ and
\begin{equation}\label{eqthsec3.2.1}
  \sum\limits_{\nu=1}^\infty \nu^{-p}E_\nu(f^{(r)})_p^p<\infty\,.
\end{equation}
Then for any $n\in\N$ we have
\begin{equation}\label{eqthsec3.2.2}
  \|f^{(r)}-T_n^{(r)}\|_p\leq C\(E_n(f^{(r)})_p+n^{1-\frac{1}{p}}\(\sum\limits_{\nu=n+1}^\infty \nu^{-p}E_\nu(f^{(r)})_p^p\)^{\frac{1}{p}}\)\,,
\end{equation}
where $T_n\in\mathcal{T}_n$, $n\in\N$, are such that $\|f-T_n\|_p=E_n(f)$ and $C$ is a constant independent of $f$ and $n$.
}
\end{theorem}

\begin{proof}
Using (\ref{eqthsec3.1.2}), we obtain
\begin{equation}\label{eqthsec3.2.3}
\begin{split}
&\sum\limits_{\nu=n+1}^\infty\nu^{rp-1}E_\nu(f)_p^p\\
&\leq C\sum\limits_{\nu=n+1}^\infty \(\nu^{-1}E_\nu(f^{(r)})_p^p+\nu^{p-2}\sum\limits_{\mu=\nu+1}^\infty \mu^{-p}E_\mu(f^{(r)})_p^p\)\\
&\leq C\sum\limits_{\nu=n+1}^\infty \nu^{p-1}\nu^{-p}E_\nu(f^{(r)})_p^p+C\(\sum\limits_{\nu=n+1}^\infty\nu^{p-2}\)\sum\limits_{\mu=n+1}^\infty \mu^{-p}E_\mu(f^{(r)})_p^p\\
&\leq Cn^{p-1}\sum\limits_{\nu=n+1}^\infty \nu^{-p}E_\nu(f^{(r)})_p^p\,.
\end{split}
\end{equation}
Therefore, by (\ref{eqthsec3.2.1}), we have that inequality (\ref{eqthsec3.A.2}) holds. Finally, combining (\ref{eqthsec3.A.2}), (\ref{eqthsec3.1.2}), and (\ref{eqthsec3.2.3}), we get (\ref{eqthsec3.2.2}).

The theorem is proved.

\end{proof}

Using Theorem~\ref{thsec3.1}, Theorem~\ref{thsec3.A}, and Theorem~\ref{thsec3.2}, we get the following equivalences.

\begin{corollary}\label{corsec3.1E}
{\it Let $0<p<1$, $r\in \N$, $\a>{1}/{p}-1$, and let $f$ be such that $f^{(r-1)}\in AC(\T)$. Then the following assertions are equivalent:

$(i)$ $E_n(f)_p=\mathcal{O}(n^{-r-\a})\,, \quad n\rightarrow
\infty\,,$

$(ii)$ $E_n(f^{(r)})_p=\mathcal{O}(n^{-\a})\,, \quad n\rightarrow
\infty\,,$

$(ii)$ $\Vert f^{(r)}-T_n^{(r)}\Vert_p=\mathcal{O}(n^{-\a})\,, \quad
n\rightarrow \infty\,,$

\noindent where $T_n\in\mathcal{T}_n$, $n\in\N$, are such that $\|f-T_n\|_p=E_n(f)$.}
\end{corollary}

\subsection{\bf Inequalities for moduli of smoothness}
Now let us consider counterparts of Theorem~\ref{thsec3.1} and Theorem~\ref{thsec3.A} in the case of periodic moduli of smoothness.

\begin{theorem}\label{th1T}
{\it Let  $0<p<1$, $k,r,m\in \N$, and let a function $f$ be such that $f^{(r-1)}\in AC(\T)$. Then for any $\d>0$ we have
\begin{equation}\label{eq.Sec2.1T}
    \w_{r+k}(f,\d)_p\le C\d^r \w_k(f^{(r)},\d)_p+C\d^{r+\frac1p-1}\(\int_0^\d
    \frac{\w_m(f^{(r)},t)_p^p}{t^{2-p}}{\rm d}t\)^\frac1p,
\end{equation}
where $C$ is a constant independent of $f$ and $\d$.
}
\end{theorem}

\begin{proof}
It is clear that we can suppose that
\begin{equation}\label{suppo}
 \int_0^1
    \frac{\w_m(f^{(r)},t)_p^p}{t^{2-p}}{\rm d}t<\infty.
\end{equation}

Let $n\in \N$ be such that $1/(n+1)<\d\le 1/n$ and let $T_n\in\mathcal{T}_n$ be polynomials of the best approximation of $f$ in $L_p(\T)$.
By~(\ref{eqM0}), we get
\begin{equation}\label{eqth1T.1}
\begin{split}
  \w_{r+k}(f,\d)_p^p&\le \w_{r+k}(f,1/n)_p^p\\
  &\le \w_{r+k}(f-T_n,1/n)_p^p+\w_{r+k}(T_n,1/n)_p^p=M_1+M_2.
\end{split}
\end{equation}
Using Lemma~\ref{lemNST}, (\ref{eqM0}), and (\ref{eqM1}), we obtain
\begin{equation}\label{eqDif5T}
\begin{split}
    M_2&\le  Cn^{-rp}\w_k (T_n^{(r)},1/n)_p^p\\
&\le Cn^{-rp}\(\Vert f^{(r)}-T_n^{(r)}\Vert_p^p+\w_k(f^{(r)},1/n)_p^p\).
\end{split}
\end{equation}
Next, by~\eqref{suppo}, Theorem~\ref{thsec3.2}, and Lemma~\ref{lemJT}, we have
\begin{equation}\label{eqDif6T}
\begin{split}
\Vert f^{(r)}-T_n^{(r)}\Vert_p^p&\le C \(\w_m(f^{(r)},1/n)_p^p+n^{p-1}\sum\limits_{\nu=n+1}^\infty \nu^{-p}\w_m(f^{(r)},1/\nu)_p^p\)\\
&\le Cn^{p-1}\int_0^{1/n} \frac{\w_m(f^{(r)},t)_p^p}{t^{2-p}}{\rm d}t.
\end{split}
\end{equation}
At the same time, by~(\ref{eqM1}),~(\ref{eqM2}),  Theorem~\ref{thsec3.1}, and Lemma~\ref{lemJT}, we derive
\begin{equation}\label{eqDif7T}
\begin{split}
M_1&\le C\Vert f-T_n\Vert_p^p\\
&\le Cn^{-rp}\(\w_m(f^{(r)},1/n)_p^p+n^{p-1}\sum\limits_{\nu=n+1}^\infty\nu^{-p}\w_m(f^{(r)},1/\nu)_p^p\) \\
&\le Cn^{p-1-rp}\int_0^{1/n} \frac{\w_m(f^{(r)},t)_p^p}{t^{2-p}}{\rm d}t.
\end{split}
\end{equation}
Thus, combining (\ref{eqth1T.1})--(\ref{eqDif7T}) and taking into
account (\ref{eqM2}) and $1/(n+1)<\d\le 1/n$, we get
(\ref{eq.Sec2.1T}).

The theorem is proved.
\end{proof}

\begin{corollary}\label{cor.Sec2.1T}
Under the conditions of Theorem~\ref{th1T}, for any $\d>0$ we have
\begin{equation*}\label{eq.cor.Sec2.1.1T}
    \w_{r+k}(f,\d)_p\leq C\d^{r+\frac{1}{p}-1}\(\int_0^\d\frac{\w_{k}(f^{(r)},t)_p^p}{t^{2-p}}{\rm d}t\)^\frac{1}{p}\,,
\end{equation*}
where $C$ is a constant independent of $f$ and $\d$.
\end{corollary}

In a similar way, combining Theorem~\ref{thsec3.A},
Lemma~\ref{lemNST}, and Lemma~\ref{lemJT} (see also the proof of
Theorem~\ref{thModFrD} below), we obtain the following inverse
inequality for the moduli of smoothness of periodic functions and
their derivatives. This result was proved earlier by Ditzian and
Tikhonov in~\cite{DiTi07}.

\begin{theorem}\label{thsec3.B}
{\it Let $f\in L_p(\T)$, $0<p<1$, $r\,,k\in \N$, and $k<r$. Then for any $\d>0$ we have
\begin{equation}\label{eqthsec3.B.1}
    \w_{r-k}(f^{(k)},\d)_p\le C\(\int_0^\d
    \frac{\w_r(f,t)_p^p}{t^{pk+1}}{\rm d}t\)^\frac1p,
\end{equation}
where $C$ is some constant independent of $f$ and $\d$. Inequality
(\ref{eqthsec3.B.1}) means that if the right-hand side is finite, then
there exists
 $f^{(k)}$ in the sense
(\ref{senseDerT}), $f^{(k)}\in L_p(\T)$, and
(\ref{eqthsec3.B.1}) holds.}
\end{theorem}

Using Theorem~\ref{th1T} and  Theorem~\ref{thsec3.B}, we get the following equivalence.
\begin{corollary}\label{corsec3.1Mod}
{\it Under the conditions of Corollary~\ref{corsec3.1E}, the
following assertions are equivalent for any $k\in \N$:

$(i)$ $\w_{r+k}(f,\d)_p=\mathcal{O}(\d^{r+\a})\,, \quad
\d\rightarrow 0\,,$

$(ii)$ $\w_{k}(f^{(r)},\d)_p=\mathcal{O}(\d^{\a})\,, \quad
\d\rightarrow 0\,.$ }
\end{corollary}

Let us consider some applications of the above theorems.

In the case $1\le p<\infty$, it is well-known the following Second Jackson theorem:
\emph{if $f\in W_p^r(\T)$, then for any $n\in \N$ one has
\begin{equation}\label{eqJsec}
  E_n(f)_p\le C_rn^{-r}\w_1\(f^{(r)},n^{-1}\)_p
\end{equation}
}(see, e.g.~\cite[p.~205]{DeLo}).
As it was mentioned above (see~\eqref{eqKop}), inequality~\eqref{eqJsec} is not true if $0<p<1$.
Combining Lemma~\ref{lemJT} and Corollary~\ref{cor.Sec2.1T}, we obtain the following counterpart of Jackson's inequality~\eqref{eqJsec} in the case $0<p<1$.

\begin{proposition}\label{propJacksec}
Let  $0<p<1$, $k,r\in \N$, and let a function $f$ be such that $f^{(r-1)}\in AC(\T)$. Then for any $n\in \N$ we have
$$
E_n(f)_p\leq C{n^{-r-\frac1p+1}}\(\int_0^{1/n}
    \frac{\w_k(f^{(r)},t)_p^p}{t^{2-p}}dt\)^\frac1p,
$$
where $C$ is a constant independent of $f$ and $n$.
\end{proposition}

At the end of this subsection, let us mention one simple application of Corollary~\ref{corsec3.1Mod}.
Recall that Krotov~\cite{Kr83} obtained the following description of functions $f\in L_p(\T)$, $0<p<1$, with the optimal rate of decreasing of $\w_1(f,h)_p$: for $f\in L_p(\T)$ we have that  $\w_1(f,h)_p=\mathcal{O}(h^{1/p})$ iff after correction on a
set of measure zero $f$ be of the form
$$
f(x)=d_0+\sum_{x_k<x}d_k,
$$
where $\{x_k\}$ is a sequence of different points from $\T$ and $\sum_{k=1}^\infty |d_k|^p<\infty$.

Using Corollary~\ref{corsec3.1Mod} and mentioned above Krotov's result, we obtain the following proposition.

\begin{proposition}\label{propDesc}
  Let $f\in L_p(\T)$, $0<p<1$, and $r\in \N$. Then $\w_r(f,h)_p=\mathcal{O}(h^{r-1+1/p})$  iff after correction on a
set of measure zero $f^{(r-1)}$ be of the form
$$
f^{(r-1)}(x)=d_0+\sum_{x_k<x}d_k,
$$
where $\{x_k\}$ is a sequence of different points from $\T$ and $\sum_{k=1}^\infty |d_k|^p<\infty$.
\end{proposition}

A sharper version of this result was obtained by another method in~\cite{Ko03} (see also~\cite[4.8.26]{TB}).


\subsection{\bf The sharpness of the main results in the periodic case}
To show the sharpness of Theorem~\ref{thsec3.1} and Theorem~\ref{th1T}, we use the next result about estimates from below for the error of the best approximation of periodic functions in $L_p(\T)$, $0<p<1$. Here, we formulate a slightly improved version of Theorem~1 from~\cite{Ko07} (see also Theorem~\ref{thbelowS}).

\begin{theorem}\label{thCT}
Let $f\in L_p(\T)$, $0<p<1$, and $s,n\in \N$. Then the following
assertions are equivalent:

\smallskip

\noindent $(i)$  for some $k>s+1/p-1$ there exist constants $M>0$ and
$\lambda>0$ such that for any $h\in (\lambda/n,1)$ one has
\begin{equation}\label{eqML}
  \w_s(f,h)_p\le M\w_k(f,h)_p,
\end{equation}

\noindent $(ii)$ there exists a constant $L>0$ such that
$$
\w_s\(f,\frac1n\)_p\le L E_n(f)_p.
$$
In particular, if inequality~\eqref{eqML} holds, then the constant $L$ depends only on $s$, $k$, $p$, $M$, and $\lambda$.
\end{theorem}

Let us consider the function
$$
\vp_\varepsilon(x)=\left\{
                   \begin{array}{ll}
                     \displaystyle \frac{x}{\varepsilon}, & \hbox{$x\in[0,\,\varepsilon)$,} \\
                     \displaystyle 1, & \hbox{$x\in[\e,\,\pi-\e)$,} \\
                     \displaystyle \frac{\pi-x}{\varepsilon}, & \hbox{$x\in[\pi-\e,\,\pi)$,} \\
                     \displaystyle 0, & \hbox{$x\in[\pi,\,2\pi)$.}
                   \end{array}
                 \right.
                 $$
Let $\vp_{\e,0}(x)=\vp_\e(x)-\frac1{2\pi}\int_0^{2\pi}\vp_\e(t){\rm d}t$ and let $f_{\e,r}(x)=I_{r-1} \vp_{\e,0}(x)$, $r=1,2\dots$, be the $r$th periodic integral of $\vp_{\e,0}(x)$, that is
$$
f_{\e,r}(x)=\int_0^x f_{\e,r-1}(t){\rm d}t+\gamma_{r-1},\quad r=2,3,\dots,
$$
where $\gamma_{r-1}$ is chosen so that $\int_{0}^{2\pi} f_{\e,r}(t){\rm d}t=0$.

One can verify that for sufficiently small $h$ and $\e$ such that $0<\e<h$
\begin{equation}\label{eqsharp1}
  \w_{r+\nu}(f_{\e,r},h)_p\asymp h^{\frac1p},\quad \nu=0,1,2,\dots,
\end{equation}
where $\asymp$ is a two-sided inequality with positive constants  independent of $\e$ and $h$.
It is also easy to see that $f_{\e,r}^{(r)}(x)=\vp'_{\e,0}(x)=\vp'_\e(x)$ and, hence,
\begin{equation}\label{eqDopMod}
  \w_1(f_{\e,r}^{(r)},h)_p\le C(p)\e^{-1}(\min\{\e,h\})^\frac1p.
\end{equation}
By Theorem~\ref{thCT} and~\eqref{eqsharp1}, there exists a constant $C=C(p,r)>0$ such that for any $n\in \N$ and sufficiently small $\e>0$
\begin{equation}\label{eqsharp2}
  E_n(f_{\e,r})_p\ge C n^{-\frac1p}.
\end{equation}
At the same time, by Lemma~\ref{lemJT} and~\eqref{eqDopMod}, we obtain for any $\g>0$
\begin{equation}\label{eqsharp3}
\begin{split}
 E_n(f_{\e,r}^{(r)})_p^p&+\sum_{\nu=n+1}^\infty \nu^{-p-\g}E_\nu(f_{\e,r}^{(r)})_p^p\\
 &\le C\(\w_1(f_{\e,r}^{(r)},n^{-1})_p^p+\sum_{\nu=n+1}^\infty \nu^{-p-\g}\w_1(f_{\e,r}^{(r)},\nu^{-1})_p^p\)\\
 &\le C\(\e^{1-p}+\e^{1-p}\sum_{\nu=n+1}^{[1/\e]}\nu^{-p-\gamma}+\e^{-p}\sum_{\nu=[1/\e]+1}^\infty \nu^{-p-\g-1}\)\\
 &\le C\e^{\min(1-p,\g)}.
\end{split}
 \end{equation}

Thus, combining (\ref{eqsharp2}) and (\ref{eqsharp3}), we get the following proposition about the sharpness of Theorem~\ref{thsec3.1}.
This proportion is also a strengthening of inequality~\eqref{eqKop}.

\begin{proposition}\label{pr1T}
  Let $0<p<1$, $r\in \N$, and $\g>0$. Then for any $B\in \R$, a constant $C>0$,
  and $n\ge n_0$ there exists a function $f_0\in C^{r-1}(\T)$ such that
  \begin{equation*}
    E_n(f_0)_p>Cn^B \(E_n(f_0^{(r)})_p+\(\sum_{\nu=n+1}^\infty \nu^{-p-\g} E_\nu (f_0^{(r)})_p^p\)^\frac1p\).
  \end{equation*}
\end{proposition}


Now we would like to show the sharpness of inequality
(\ref{eq.Sec2.1T}). It follows from \cite[p.~188]{PePo} that (\ref{eq.Sec2.1T})
does not hold without the integral in the right-hand of the inequality. By using the same arguments as
after Theorem~\ref{thCT}, we can prove the following stronger assertion.

\begin{proposition}\label{pr.Sec2.1}
 Let $0<p<1$, $r\in \N$, and $\g>0$. Then for any $B\in \R$, a constant $C>0$,
  and $\d\in (0,\d_0)$ there exists a function $f_0\in C^{r-1}(\T)$ such that
$$
\w_{r+k}(f_0,\d)_p>C\d^{B}\(\w_k(f_0^{(r)},\d)_p^p+\int_0^\d
    \frac{\w_k(f_0^{(r)},t)_p^p}{t^{2-p-\g}}{\rm d}t\)^\frac1p\,.$$
\end{proposition}

In particular, the above proposition implies that if $\g>0$, then the following inequality
\begin{equation*}
\label{eq.Sec1.3TS}
    \w_{r+k}(f,\d)_p\leq C\d^{r+\frac{1}{p}-1}\(\int_0^\d\frac{\w_{k}(f^{(r)},t)_p^p}{t^{2-p-\g}}{\rm d}t\)^\frac{1}{p}
\end{equation*}
does not
hold for all $f\in C^{r-1}(\T)$ with the constant $C$ independent of $f$ and $\d$ (cf.~Corollary~\ref{cor.Sec2.1T}).


\section{Approximation of functions by splines. Non-periodic moduli of smoothness}

In this section, we consider the case of approximation of
functions by splines in the spaces $L_p[0,1]$ with the (quasi-)norm $\Vert \cdot \Vert_p=\Vert \cdot \Vert_{L_p[0,1]}$.

Denote by $\mathcal{S}_{m,n}$ the set of all spline functions of
degree $m-1$ with the knots $t_j=t_{j,n}:=j/n$, $j=0,\ldots,n,$ i.e.
$S\in \mathcal{S}_{m,n}$ if $S\in C^{m-2}[0,1]$ and $S$ is some
algebraic polynomial of degree $m-1$ in each interval
$(t_{j-1},\,t_j)$, $j=1,\ldots,n$. Recall that $\mathcal{P}_r$ denotes the set of all algebraic polynomials of degree at most $r$.

Let
$$
\E_{m,n}(f)_p=\inf_{S\in \Sp_{m,n}}\Vert f-S\Vert_p
$$
denote the error of the best approximation of a function $f$ by
splines $S\in \Sp_{m,n}$ in $L_p[0,1]$.


\subsection{\bf Inequalities for the error of the best approximation of functions by splines}

The following result is a counterpart of Theorem~\ref{thsec3.1} in the case of approximation of functions by splines in the spaces $L_p[0,1]$.

\begin{theorem}\label{th1S}
Let $0<p<1$, $m$, $r\in \N$, $r<m$, and let $f$ be such that
$f^{(r-1)}\in AC[0,1]$ and
\begin{equation*}
  \sum\limits_{\nu=1}^\infty \nu^{-p}\E_{m-r,\nu}(f^{(r)})_p^p<\infty\,.
\end{equation*}
Then for any $n\in \N$ we have
\begin{equation*}
  \E_{m,n}(f)_p\leq\frac{C}{n^r}\(\E_{m-r,n}(f^{(r)})_p+n^{1-\frac{1}{p}}\(\sum\limits_{\nu=n+1}^\infty\nu^{-p}\E_{m-r,\nu}(f^{(r)})_p^p\)^{\frac{1}{p}}\)\,,
\end{equation*}
where $C$ is a constant independent of $f$ and $n$.
\end{theorem}

The next three lemmas are the main tools for proving Theorem~\ref{th1S} as well as other results in this section. These lemmas are analogues of
Lemmas~\ref{lemJT}--\ref{lemNikT} for splines.
The first lemma is the Jackson-type theorem  (see~\cite{Os80}, see also~\cite[Ch. 2]{DeLo}).

\begin{lemma}\label{lemJS}
Let $f\in L_p[0,1]$, $0<p<1$, and $r$, $n\in \N$. Then
there exists a spline $S_n\in \mathcal{S}_{r,n}$ such that
$$
\|f-S_n\|_p\leq C\omega_{r}(f,n^{-1})_p\,,
$$
where $C$ is a constant independent of $f$ and $n$.
\end{lemma}
The second lemma gives  equivalences for moduli of smoothness of splines
in the spaces $L_p[0,1]$ (see~\cite{Kop06} and~\cite{HuYu}).

\begin{lemma}\label{lemNSS}
Let $0<p<\infty$, $l\ge 2$, and $1\leq k\leq l$. Then for any $S_n\in \mathcal{S}_{l,n}$, $n\in \N$, we have
$$
n^{-\nu }\omega_{k-\nu}(S_n^{(\nu)},n^{-1})_p\asymp
\omega_{k}(S_n,n^{-1})_p, \quad 1\leq\nu \leq\min(k,\,l-1)\,,
$$
where $\asymp$ is a two-sided inequality with positive constants depending
only on $l$ and $p$.
\end{lemma}
We also need the following Nikolskii type inequality for splines (see~\cite{Os80} or~\cite[Ch.~5]{DeLo}).

\begin{lemma}\label{lemNikS}
Let  $0<p\le q<\infty$ and $r\in \N$. Then for any $S_n\in \mathcal{S}_{r,n}$, $n\in \N$, we have
$$
\|S_n\|_q\leq Cn^{\frac{1}{p}-\frac{1}{q}}\|S_n\|_p\,,
$$
where $C$ is a constant depending only on $q$ and $r$.
\end{lemma}

\begin{proof}\emph{of\,\, Theorem~\ref{th1S}\,\,\,}
To prove the theorem one can use Lemmas~\ref{lemJS}--\ref{lemNikS} and repeat the scheme of proving
Theorem~\ref{thsec3.1}. We only note that the prove of
Theorem~\ref{th1S} is slightly simpler than the proof of
Theorem~\ref{thsec3.1}, since there is no need to perform
additional technical steps concerning using the function
$\tau_u$ and the polynomials $T_n$ as in the proof of
Theorem~\ref{thsec3.1}. See, for example, the proof of Theorem~\ref{th3}, in which a similar situation is considered.
\end{proof}

Now we are going to obtain an analog of inequality
(\ref{eqthsec3.A.2}) for the error $\E_{m,n}(f)_p$.
For this purpose, we need the notion of a derivative in the sense of $L_p[0,1]$.
By analogy with the corresponding definition~\eqref{senseDerT}, we define the derivative of $f\in L_p[0,1]$ as a function
$g$ satisfying
\begin{equation}\label{eqProizvLp}
    \bigg\Vert \frac{\D_h^k f}{h^k}-g\bigg\Vert_{L_p[0,\,1-kh]}\to
    0\quad\text{as}\quad h\to 0+\,.
\end{equation}
In this case, we write $g=f^{(k)}$.

\begin{theorem}\label{th2S}
Let $f\in L_p[0,1]$, $0<p<1$, $m\in\N$, and let for some $k\in \N$, $k<m$, one has
\begin{equation}\label{eqth2S.1}
\sum\limits_{\nu=1}^\infty \nu^{kp-1}\E_{m,\nu}(f)_p^p<\infty\,.
\end{equation}
Then $f$ has the derivative $f^{(k)}$ in the sense of $L_p[0,1]$ and for any $n\in\N$
\begin{equation}\label{eqth2S.2}
\|f^{(k)}-S_n^{(k)}\|_p\leq C\(n^k
\E_{m,n}(f)_p+\(\sum\limits_{\nu=n+1}^\infty
\nu^{kp-1}\E_{m,\nu}(f)_p^p\)^{\frac{1}{p}}\)\,,
\end{equation}
where $S_n\in\mathcal{S}_{m,n}$ is such that $\|f-S_n\|_p=\E_{m,n}(f)$ and  $C$ is a constant independent of $n$ and $f$.
\end{theorem}

Surprisingly Theorem~\ref{th2S} as well as Theorem~\ref{thModFrD}
below are new.
To prove Theorem~\ref{th2S}, we need the following two auxiliary results.
The first one is an analog of Markov's inequality (see \cite[p.~136]{DeLo}).
\begin{lemma}\label{lemMarkS}
Let $0<p<\infty$, $m,r\in \N$, and  $r< m$. Then for any
$S_n\in\mathcal{S}_{m,n}$, $n\in\N$,
\begin{equation*}
\Vert S_n^{(r)}\Vert_p\leq C n^{r}\Vert S_n\Vert_p,
\end{equation*}
where $C$ is a constant independent of $S_n$.
\end{lemma}

The next auxiliary result is an analog of Theorem~2.3 from~\cite{DiTi07} in the case of approximation of functions by splines in the spaces $L_p[0,1]$, $0<p<1$.

\begin{lemma}\label{lem.th2}
Let $f\in L_p[0,1]$, $0<p<1$, $1\le k<r$. Suppose $S_n \in\mathcal{S}_{r,n}$, $n\in \N$, are such that
\begin{equation*}
  \Vert f-S_n\Vert_p=o(n^{-k})\quad\text{and}\quad \Vert g-S_n^{(k)}\Vert_p=o(1)\quad\text{as}\quad n\to\infty.
\end{equation*}
Then
$f^{(k)}=g$, that is $g$ satisfies (\ref{eqProizvLp}).
\end{lemma}

\begin{proof}
For any sufficiently small $\e>0$, we choose $n_0=n_0(\e)$ such that for $n\geq n_0$
\begin{equation}\label{eq.lem.th2.2}
  \Vert f-S_n\Vert_p\leq\e n^{-k}\quad\text{and}\quad \Vert g-S_n^{(k)}\Vert_p\leq\e.
\end{equation}
Let $h=\e^{1/2k}n^{-1}$.
We have
\begin{equation}\label{eq.lem.th2.3}
\begin{split}
\bigg\Vert \frac{\D_h^k f}{h^k}-g\bigg\Vert_{L_p[0,\,1-kh]}^p&\leq \bigg\Vert \frac{\D_h^k (f-S_n)}{h^k}\bigg\Vert_{L_p[0,\,1-kh]}^p\\
&+
\bigg\Vert \frac{\D_h^k S_n}{h^k}-S_n^{(k)}\bigg\Vert_{L_p[0,\,1-kh]}^p+\Vert g-S_n^{(k)}\Vert_{p}^p\\
&=J_1+J_2+J_3\,.
\end{split}
\end{equation}
By (\ref{eq.lem.th2.2}), we get
\begin{equation}\label{eq.lem.th2.3.5}
J_1\leq2^kh^{-kp}\Vert f-S_n\Vert_{p}^p\leq 2^k\e^{p/2}\quad \textrm{{and}} \quad J_3\leq\e^p\,.
\end{equation}
To estimate $J_2$, we use the following representing formula for a spline $S_n\in \mathcal{S}_{r,n}$
\begin{equation}\label{eq.lem.th2.4}
S_n(x)=P(x)+\sum_{j=1}^{n-1}a_j(x-t_j)_+^{r-1}\,,
\end{equation}
where $P\in \mathcal{P}_{r-1}$, $x_+=x$ if $x\geq0$ and $x_+=0$ if $x<0$ (see \cite{Os80}).
Recall also that
\begin{equation}\label{eq.lem.th2.6}
\w_r(S_n,\,n^{-1})_p^p\asymp
n^{-(1+(r-1)p)}\sum\limits_{j=1}^{n-1}|a_j|^p\,,
\end{equation}
where $\asymp$ is a two-sided inequality with positive constants independent of $S_n$.
Equivalence (\ref{eq.lem.th2.6}) follows from Lemma~2.1 in
\cite{HuYu}, which was
proved only in the case $1\le p<\infty$. It is easy to verify that this lemma is valid in the case $0<p<1$, too.

Now, using (\ref{eq.lem.th2.4}), we get for some fixed
$j\in[1,n-1]$ that
\begin{equation}\label{eq.lem.th2.7}
S_n(x)=\left\{
                   \begin{array}{ll}
                     \displaystyle P_{j}(x), & \hbox{$x\in(t_{j-1},\,t_{j}]$,} \\
                     \displaystyle a_{j}(x-t_{j})^{r-1}+P_{j}(x), & \hbox{$x\in(t_{j},\,t_{j+1})$,}
                   \end{array}
                 \right.
\end{equation}
where $P_{j}(x)\in\mathcal{P}_{r-1}\,.$
Hence, for some $l\in[1,k]$ and
$x\in(t_{j}-lh,\,t_{j}-(l-1)h)$, we obtain
\begin{equation}\label{eq.lem.th2.9}
\begin{split}
\D_h^k S_n(x)&=a_{j}\D_h^k(x-t_{j})_+^{r-1}+\D_h^kP_{j}(x)\\
&=a_{j}\sum_{s=l}^k\binom{k}{s}(-1)^s(x+sh-t_{j})^{r-1}+\D_h^kP_{j}(x)\,.
\end{split}
\end{equation}
Next, taking into account (\ref{eq.lem.th2.9}) and (\ref{eq.lem.th2.7}), we derive
\begin{equation}\label{eq.lem.th2.10}
\begin{split}
J_2&=\sum\limits_{j=1}^{{n}-1}\left\{\int_{t_{j-1}}^{t_j-rh}+\sum_{l=1}^r\int_{t_j-lh}^{t_j-(l-1)h}\right\}\left|\frac{\D_h^k S_n(x)}{h^k}-S_n^{(k)}(x)\right|^p{\rm d}x\\
&=\sum\limits_{j=1}^{{n}-1}\(I_{0,j}+\sum_{l=1}^rI_{l,j}\)\,.
\end{split}
\end{equation}
By (\ref{eq.lem.th2.7}), for $I_{0,j}$, $j\in [1,n-1]$, we have
$$
I_{0,j}=\bigg\Vert \frac{\D_h^k P_j}{h^k}-P_j^{(k)}\bigg\Vert_{L_p[t_{j-1},\,t_{j}-rh]}^p\,.
$$
By formula (2.6) in \cite{HuLi05}, we get
$$
\D_h^k P_j(x)=\sum_{\nu=0}^{r-1-k}\frac{P_j^{(k+\nu)}(x)}{\nu!}h^{k+\nu}\xi_{k+\nu}^\nu,\quad 0<\xi_{k+\nu}<k\,.
$$
Therefore,
\begin{equation}\label{eq.lem.th2.11}
\begin{split}
I_{0,j}&=\bigg\Vert \sum_{\nu=1}^{r-1-k}\frac{P_j^{(k+\nu)}}{\nu!}h^{k+\nu}\xi_{k+\nu}^\nu\bigg\Vert_{L_p[t_{j-1},\,t_{j}-rh]}^p\\
&\leq C\sum_{\nu=1}^{r-1-k}h^{(k+\nu)p}\|S_n^{(k+\nu)}\|_{L_p[t_{j-1},\,t_{j}]}^p\\
&\leq C \e^{p/2}\sum_{\nu=1}^{r-1-k}n^{-(k+\nu)p}\|S_n^{(k+\nu)}\|_{L_p[t_{j-1},\,t_{j}]}^p\,.
\end{split}
\end{equation}

Now, let us consider $I_{l,j}$ for $j\in[1,n-1]$ and $l\in[1,r]$. By  (\ref{eq.lem.th2.9}) and (\ref{eq.lem.th2.7}), we have
\begin{equation}\label{eq.lem.th2.12}
\begin{split}
I_{l,j}&=\bigg\Vert a_jh^{-k}\D_h^k(\cdot-t_j)_+^{r-1}+\frac{\D_h^kP_j}{h^k}-P_j^{(k)}\bigg\Vert_{L_p[t_{j}-lh,\,t_{j}-(l-1)h]}^p\\
&\leq
|a_j|^ph^{-kp}\int_{t_j-lh}^{t_j-(l-1)h}|\D_h^k(x-t_j)_+^{r-1}|^p{\rm d}x\\
&\quad\quad\quad\quad\quad\quad\quad\quad\quad+\bigg\Vert
\frac{\D_h^k
P_j}{h^k}-P_j^{(k)}\bigg\Vert_{L_p[t_{j}-lh,\,t_{j}-(l-1)h]}^p\,.
\end{split}
\end{equation}
Applying the same estimates as for $I_{0,j}$ in~\eqref{eq.lem.th2.11} to the second summand in (\ref{eq.lem.th2.12}) and
taking into account that
$$
\int_{t_j-lh}^{t_j-(l-1)h}|\D_h^k(x-t_j)_+^{r-1}|^p{\rm d}x=h^{1+(r-1)p}\int_0^1|\D_1^k(x-l)_+^{r-1}|^p{\rm d}x,
$$
we obtain
\begin{equation}\label{eq.lem.th2.13}
I_{l,j}\leq C\e^{p/2}\(n^{-(1+(r-k-1)p)}|a_j|^p+
\sum_{\nu=1}^{r-1-k}n^{-(k+\nu)p}\|S_n^{(k+\nu)}\|_{L_p[t_{j-1},\,t_{j}]}^p\)\,.
\end{equation}

Combining (\ref{eq.lem.th2.10}), (\ref{eq.lem.th2.11}), (\ref{eq.lem.th2.13}), and applying (\ref{eq.lem.th2.6}), \eqref{eqM1}, and Lemmas~\ref{lemMarkS} and~\ref{lemNSS}, we derive
\begin{equation}\label{eq.lem.th2.14}
\begin{split}
J_2&\leq C\e^{p/2}\(n^{-(1+(r-k-1)p)}\sum_{j=1}^{n-1}|a_j|^p+\sum_{\nu=1}^{r-1-k}n^{-(k+\nu)p}\|S_n^{(k+\nu)}\|_{p}^p\)\\
&\leq C\e^{p/2}\(n^{kp}\w_r(S_n,\,n^{-1})_p^p+\|S_n\|_p^p\)\\
&\leq C\e^{p/2}\(n^{kp}\w_k(S_n,\,n^{-1})_p^p+\|f-S_n\|_p^p+\|f\|_p^p\)\\
&\leq C\e^{p/2}\(\Vert S_n^{(k)}\Vert_p^p+\e^p+\|f\|_p^p\)\\
&\leq C\e^{p/2}\(\Vert g\Vert_p^p+\e^p+\|f\|_p^p\)\,.
\end{split}
\end{equation}
Finally, from (\ref{eq.lem.th2.3}), (\ref{eq.lem.th2.3.5}), and (\ref{eq.lem.th2.14}), we get
$$\bigg\Vert \frac{\D_h^k f}{h^k}-g\bigg\Vert_{L_p[0,1-kh]}^p\leq C\e^{p/2}\(\e^{p/2}+\Vert g\Vert_p^p+\|f\|_p^p\)\,.$$
Since the right-hand side of the above inequality does not depend on $S_n$, we have that $g=f^{(k)}$.

The lemma is proved.
\end{proof}

\begin{proof}\emph{of\,\, Theorem~\ref{th2S}\,\,\,}
Let $N\in \N$ be such that $2^{N-1}\le n<2^N$. Assuming for a moment
that $f^{(k)}$ exists, we get
\begin{equation}\label{eqth2S1}
  \Vert f^{(k)}-S_n^{(k)}\Vert_p^p\le \Vert f^{(k)}-S_{2^N}^{(k)}\Vert_p^p+\Vert S_{2^N}^{(k)}-S_n^{(k)}\Vert_p^p.
\end{equation}
By Lemma~\ref{lemMarkS}, we obtain
\begin{equation}\label{eqth2S2}
\begin{split}
  \Vert S_{2^N}^{(k)}-S_n^{(k)}\Vert_p^p\le C2^{kpN}  \Vert S_{2^N}-S_n\Vert_p^p
  \le Cn^{kp} \E_{m,n}(f)_p^p
\end{split}
\end{equation}
and
\begin{equation}\label{eqth2S3}
\begin{split}
\sum_{\nu=N}^\infty \Vert S_{2^{\nu+1}}^{(k)}-S_{2^\nu}^{(k)}\Vert_p^p  &\le
C \sum_{\nu=N}^\infty 2^{kp\nu} \Vert
S_{2^{\nu+1}}-S_{2^\nu} \Vert_p^p\\
&\le C \sum_{\nu=N}^\infty 2^{kp\nu}\E_{m,2^\nu}(f)_p^p\,.
\end{split}
\end{equation}
Thus, by the completeness of $L_p[0,1]$ and condition
(\ref{eqth2S.1}), there exists a function $g\in L_p[0,1]$ such that
\begin{equation}\label{eqprthModFrD.7}
\begin{split}
\Vert g-S_{2^N}^{(k)}\Vert_p=\lim_{l\to\infty}\Vert
S_{2^l}^{(k)}-S_{2^N}^{(k)}\Vert_p\le C \(\sum_{\nu=N}^\infty
2^{kp\nu}\E_{m,2^\nu}(f)_p^p\)^\frac 1p.
\end{split}
\end{equation}
In the above inequality, we use the equality $$
S_{2^l}-S_{2^N}=\sum\limits_{\nu=N}^{l-1}(S_{2^{\nu+1}}-S_{2^\nu})$$ and~\eqref{eqth2S3}. It is also easy to see that
\begin{equation}\label{eqprthModFrD.8}
\begin{split}
\Vert f-S_{2^N}\Vert_p\leq C 2^{-N k}\(2^{N
k}\E_{m,2^N}(f)_p\)=o(2^{-N k})\quad \mathrm{as}\quad N\to \infty.
\end{split}
\end{equation}
Therefore, by Lemma~\ref{lem.th2},  (\ref{eqprthModFrD.8}), and (\ref{eqprthModFrD.7}),  we obtain that $g=f^{(k)}$.
Finally, combining (\ref{eqth2S1}), (\ref{eqth2S2}), and (\ref{eqprthModFrD.7}), we get (\ref{eqth2S.2}).

Theorem~\ref{th2S} is proved.
\end{proof}

By analogy with the case of approximation of functions by trigonometric polynomials,
combining Theorem~\ref{th1S} and Theorem~\ref{th2S}, we obtain the
following result about the simultaneous approximation of a function and its derivatives by splines in
$L_p[0,1]$ for $0<p<1$ (see also Theorem~7.4 in~\cite{DeLo} for the
case $p\ge 1$).

\begin{theorem}\label{thsec3.2S}
{\it Let $0<p<1$, $r,m\in \N$, $r<m$, and let $f$ be such that
$f^{(r-1)}\in AC[0,1]$ and
\begin{equation*}
  \sum\limits_{\nu=1}^\infty \nu^{-p}\E_{m-r,\nu}(f^{(r)})_p^p<\infty\,.
\end{equation*}
Then for any $n\in\N$ we have
\begin{equation*}
  \|f^{(r)}-S_n^{(r)}\|_p\leq C\(\E_{m-r,n}(f^{(r)})_p+n^{1-\frac{1}{p}}
  \(\sum\limits_{\nu=n+1}^\infty \nu^{-p}\E_{m-r,\nu}(f^{(r)})_p^p\)^{\frac{1}{p}}\)\,,
\end{equation*}
where $S_n\in\mathcal{S}_{m,n}$, $n\in\N$, are such that
$\|f-S_n\|_p=\E_{m,n}(f)$ and $C$ is a constant independent of $f$ and $n$.
}
\end{theorem}

Combining Theorems~\ref{th1S},~\ref{th2S}, and~\ref{thsec3.2S}, we obtain the following equivalences.

\begin{corollary}\label{corsec3.1}
{\it Let $0<p<1$, $m,r\in \N$, $r<m$, $\a>{1}/{p}-1$, and let $f$
be such that $f^{(r-1)}\in AC[0,1]$. Then the following assertions
are equivalent:

$(i)$ $\E_{m,n}(f)_p=\mathcal{O}(n^{-r-\a})\,, \quad n\rightarrow
\infty\,,$

$(ii)$ $\E_{m-r,n}(f^{(r)})_p=\mathcal{O}(n^{-\a})\,, \quad
n\rightarrow \infty\,,$

$(ii)$ $\Vert f^{(r)}-S_n^{(r)}\Vert_p=\mathcal{O}(n^{-\a})\,, \quad
n\rightarrow \infty\,,$

\noindent where $S_n\in\mathcal{S}_{m,n}$, $n\in\N$, are such that
$\|f-S_n\|_p=\E_{m,n}(f)$.}
\end{corollary}

\subsection{\bf Inequalities for the non-periodic moduli of smoothness}

By analogy with the proof of Theorem~\ref{th1T}, combining Theorem~\ref{th1S}, Lemma~\ref{lemNSS},
and Lemma~\ref{lemJS}, we obtain a
non-periodic analogue of Theorem~\ref{th1T}.

\begin{theorem}\label{th1Sm}
Let $0<p<1$, $r,k,m\in \N$, and let a function $f$ be such that $f^{(r-1)}\in
AC[0,1]$. Then for any $\d>0$ we have
\begin{equation*}
    \w_{r+k}(f,\d)_p\le C\d^r \w_k(f^{(r)},\d)_p+C\d^{r+\frac1p-1}\(\int_0^\d
    \frac{\w_m(f^{(r)},t)_p^p}{t^{2-p}}{\rm d}t\)^\frac1p,
\end{equation*}
where $C$ is some constant independent of $f$ and $\d$.
\end{theorem}

The next theorem is a counterpart of Theorem~\ref{thsec3.B}. Under more restrictive conditions, this theorem was obtained in~\cite{Ko03}.

\begin{theorem}\label{thModFrD}
{\it Let $f\in L_p[0,\,1]$, $0<p<1$, $r\,,k\in \N$, and $k<r$. Then for any $\d>0$ we have
\begin{equation}\label{eqthModFrD.1}
    \w_{r-k}(f^{(k)},\d)_p\le C\(\int_0^\d
    \frac{\w_r(f,t)_p^p}{t^{pk+1}}{\rm d}t\)^\frac1p,
\end{equation}
where $C$ is some constant independent of $f$ and
$\d$. Inequality (\ref{eqthModFrD.1}) means that if the right-hand side is finite, then there exists
 $f^{(k)}$ in the sense
(\ref{eqProizvLp}), $f^{(k)}\in L_p[0,\,1]$, and
(\ref{eqthModFrD.1}) holds.}
\end{theorem}

\begin{proof}
The theorem can be proved combining Theorem~\ref{th2S}, Lemma~\ref{lemNSS}, and Lemma~\ref{lemJS}.
\end{proof}

\begin{corollary}\label{corsec3.1ModS}
Under the conditions of Corollary~\ref{corsec3.1}, the following
assertions are equivalent for any $k\in \N$:

$(i)$ $\w_{r+k}(f,\d)_p=\mathcal{O}(\d^{r+\a})\,, \quad
\d\rightarrow 0\,,$

$(ii)$ $\w_{k}(f^{(r)},\d)_p=\mathcal{O}(\d^{\a})\,, \quad
\d\rightarrow 0\,.$
\end{corollary}

\subsection{\bf The sharpness of the main results in the non-periodic case}

We omit the formulation of analogues of Proposition~\ref{pr1T} and
\ref{pr.Sec2.1}. We only note that the sharpness of
Theorems~\ref{th1S} and~\ref{th1Sm} can be shown by using the same
examples as in the periodic case in Section~2. For this, one can use the following counterpart
of Theorem~\ref{thCT}. The next theorem is interesting in its own (see, e.g.,~\cite{Rath94}).

\begin{theorem}\label{thbelowS}
Let $f\in L_p[0,1]$, $0<p<1$, and $s,n\in \N$. Then the following
assertions are equivalent:

\smallskip

\noindent $(i)$ for some $k>s$ there exist constants $M>0$ and
$\mu>0$ such that for any $h\in (\mu/n,1)$
\begin{equation}\label{eq.thbelowS1}
    \w_s(f,h)_p\le M\w_k(f,h)_p,
\end{equation}

\noindent $(ii)$  there exists a constant $L>0$ independent of $f$
and $n$ such that
\begin{equation}\label{eq.thbelowS2}
\w_s\(f,\frac1n\)_p\le L \E_{k,n}(f)_p.
\end{equation}
\end{theorem}

\begin{proof}
Let condition~(\ref{eq.thbelowS1}) be satisfied. Then from
(\ref{eqM1}) and (\ref{eqM2}) we get
\begin{equation}\label{rathoreqiP}
      \w_{k}(f,\lambda h)_{p}\le C M (\lambda+1)^{s+\frac
      1p-1}\w_{k}(f,h)_{p}.
\end{equation}
Using Lemma~\ref{lemJS} and inequality (\ref{rathoreqiP}), we obtain
\begin{equation}\label{rathInvP}
\begin{split}
      \frac{1}{n^{(k-1)p+1}}\sum_{\nu=1}^n\nu^{(k-1)p}\E_{k,\nu}(f)_{p}^p&\le
      \frac{C}{n^{(k-1)p+1}}\sum_{\nu=1}^n\nu^{(k-1)p} \w_{k}\left(f,\frac
      {1}{\nu}\right)_{p}^p\\
      &\le\frac{CM^p}{n^{(k-s)p}} \w_{k}\left(f,\frac 1n
      \right)^p_{p}\sum_{\nu=1}^n\nu^{(k-s)p-1}\\
      &\le CM^p \w_{k}\left(f,\frac 1n\right)_{p}^p,
\end{split}
\end{equation}
where $C$ is some positive constant independent of $f$ and $n$.
Next, using the inverse inequality for the best spline
approximation (see~\cite[Theorem~3]{Os80}) and (\ref{rathInvP}), we
get for all $m, n\in\N$
\begin{equation}\label{eqRathorXXXP}
\begin{split}
       \w_{k}\(f,\frac{1}{mn} \)_{p}^p&\le
       \frac{C}{(mn)^{(k-1)p+1}}\sum_{\nu=1}^{mn}\nu^{(k-1)p}\E_{k,\nu}(f)_{p}^p\\
       &=\frac{C}{(mn)^{(k-1)p+1}}\(\sum_{\nu=n+1}^{mn}\nu^{(k-1)p}\E_{k,\nu}(f)_{p}^p+
      \sum_{\nu=1}^n\nu^{(k-1)p}\E_{k,\nu}(f)_{p}^p \)\\
       &\le  C\Bigg(\frac{1}{(mn)^{(k-1)p+1}}\sum_{\nu=n+1}^{mn}
      \nu^{(k-1)p}\E_{k,\nu}(f)_{p}^p+\frac{M^p}{m^{(k-1)p+1}}
      \w_{k}\left(f,\frac 1n\right)_{p}^p\Bigg).
\end{split}
\end{equation}
Inequality~(\ref{eqRathorXXXP}) implies that
\begin{equation}\label{XXX}
\begin{split}
        \sum_{\nu=n+1}^{mn}\nu^{(k-1)p}\E_{k,\nu}(f)_{p}^p\ge
      \frac{(mn)^{(k-1)p+1}}{C}&\w_{k}\left(f,\frac{1}{mn}\right)_{p}^p\\
      &-M^p n^{(k-1)p+1}\w_{k}\left(f,\frac
     1n\right)_{p}^p\,.
\end{split}
\end{equation}
From~\eqref{XXX}, using the monotonicity of $\E_{k,n}(f)_{p}$ and
(\ref{rathoreqiP}), we derive
\begin{equation*}
      \E_{k,n}(f)_{p}^p\sum_{\nu=n+1}^{mn}\nu^{(k-1)p}\ge (C m^{(k-s)p}-M^p)n^{(k-1)p+1}
      \w_{k}\left(f,\frac{1}{n} \right)_{p}^p.
\end{equation*}
Thus,  choosing an appropriate $m$, we can find a positive constant
$C$ independent of $f$ and $n$ such that
\begin{equation*}
      \E_n(f)_{p}^p\ge C \w_{k}\left(f,\frac 1n\right)_{p}^p.
\end{equation*}
From the last inequality and (\ref{eq.thbelowS1}) we obtain
(\ref{eq.thbelowS2}).

The reverse direction is an immediate consequence of
Lemma~\ref{lemJS}.
\end{proof}

\section{Approximation of functions by algebraic polynomials. Ditzian-Totik moduli of smoothness}

In this section, we denote $\Vert \cdot \Vert_p=\Vert \cdot \Vert_{L_p[-1,1]}$.
Let $\mathcal{P}_n$ be the set of all algebraic polynomials of degree at most $n$ and let
$$
E_n(f)_{p,w}=\inf_{P\in \mathcal{P}_n}\Vert w(f-P)\Vert_p
$$
be the error of the best approximation of a function $f$ in the space $L_p[-1,1]$ with a weight $w$.
In what follows, if $w_0(x)=1$, then we write
$$
E_n(f)_{p}=E_n(f)_{p,w_0}.
$$

Let $f\in L_p[-1,1]$, $0<p<\infty$, $r\in\N$, $\vp(x)=\sqrt{1-x^2}$, and $w(x)=\vp^\s(x)$, $\s\ge 0$. Recall that the Ditzian-Totik modulus of smoothness $\w_r^\vp(f,\d)_{p}$ and the weighted main part moduli of smoothness $\Omega_r^\vp(f,\d)_{p,w}$ are given by
$$
\w_r^\vp (f,\d)_{p}=\sup_{|h|\le \d}\Vert \Dl_{h\vp}^r f\Vert_{L_p[-1,1]}
$$
and
$$
\Omega_r^\vp(f,\d)_{p,w}=\sup_{|h|\le \d}\Vert w\Dl_{h\vp}^r
f\Vert_{L_p[-1+2r^2h^2,1-2r^2h^2]},
$$
where
$$
\Dl_{h\vp(x)}^r f(x)=\left\{
                   \begin{array}{ll}
                     \displaystyle \sum_{k=0}^r (-1)^k\binom{r}{k}f\(x+\(\frac r2-k\)h\vp(x)\), & \hbox{$x\pm \frac r2 h\vp(x)\in [-1,1]$,} \\
                     \displaystyle 0, & \hbox{otherwise.}
                   \end{array}
                 \right.
$$


\subsection{{\bf Inequalities for the error of  the best approximation of functions by algebraic polynomials}}

The next theorem is an analogue of Theorem~\ref{thsec3.1} and Theorem~\ref{th1S} in the case of approximation of functions by algebraic polynomials.

\begin{theorem}\label{th3}
Let $0<p<1$, $r\in\N$, and let $f$ be such that $f^{(r-1)}\in AC[-1,\,1]$, $\vp^rf^{(r)}\in L_1[-1,1]$, and
\begin{equation}\label{eq.th3.0}
\sum_{\nu=r+1}^\infty\nu^{1-2p}E_{\nu-r}(f^{(r)})_{p,\vp^r}^p<\infty.
\end{equation}
Then for any $n\ge n(r)$ we have
\begin{equation}\label{eq.th3.1}
E_n(f)_p\leq \frac{C}{n^r}\(E_{n-r}(f^{(r)})_{p,\,\vp^r}+n^{2-\frac{2}{p}}\(\sum_{\nu=n+1}^\infty\nu^{1-2p}E_{\nu-r}(f^{(r)})_{p,\vp^r}^p\)^{\frac{1}{p}}\)\,,
\end{equation}
where $C$ is a constant independent of $f$ and $n$.
\end{theorem}

As in the previous sections, the proof of Theorem~\ref{th3} is based on the next three lemmas.
The first one is the Jackson-type theorem (see~\cite{DLY} for the case $0<p<1$ and~\cite{DiTo} for the case $p\ge 1$).

\begin{lemma}\label{lemJP}
Let $f\in L_p[-1,1]$, $0< p<\infty$, $k\in \N$, and $n>k$. Then
\begin{equation*}
  E_n(f)_p\le C\w_k^\vp\(f,\frac 1n\)_p,
\end{equation*}
where $C$ is a constant independent of $n$ and $f$.
\end{lemma}

The following lemma was proved in~\cite{HuLi05}.

\begin{lemma}\label{lemNSP}
Let $0<p<1$, $r,n\in \N$, and $0<\d\le (Mn)^{-1}$. Then for any $P_n\in \mathcal{P}_n$ we have
$$
\w_r^\vp(P_n,\d)_p\asymp
\d^k\Omega_{r-k}^\vp(P_n^{(k)},\d)_{p,\vp^k}\asymp \d^r \Vert \vp^r
P_n^{(r)}\Vert_p,\quad 0\le k<r,
$$
where $\asymp$ is a two-sided inequality with positive constants independent of $\d$ and $P_n$ and $M$ is some constant depending only on $r$ and $p$.
\end{lemma}

We also use the following Nikolskii type inequality for algebraic polynomials (see~\cite{DiTi05}).

\begin{lemma}\label{lemNikP}
Let  $0<p\le q<\infty$ and $r,n\in \N$. Then for any $P_n\in \mathcal{P}_{n}$ we have
$$
\|\vp^r P_n\|_q\leq Cn^{2\(\frac{1}{p}-\frac{1}{q}\)}\|\vp^r P_n\|_p\,,
$$
where  $C$ is a constant independent of $P_n$.
\end{lemma}

To prove Theorem~\ref{th3}, we also need analogues of inequalities
(\ref{eqM1}) and (\ref{eqM2}) for the Ditzian-Totik moduli of
smoothness. The first one is a standard inequality for
moduli of smoothness of different orders
\begin{equation}\label{eq.Sec3.M1}
       \w_k^\vp(f,\d)_p\le C\w_r^\vp(f,\d)_p\le C\Vert f\Vert_p,\quad k\ge r,\quad \d>0,
\end{equation}
(see, e.g.,~\cite[inequality (6.8)]{DHI}). Concerning inequality (\ref{eqM2}), it is only known  that
\begin{equation}\label{eq.*}
\w_r^{\vp}(f,\,2t)_p\leq C\w_r^{\vp}(f,t)_p
\end{equation}
(see~\cite[inequality (5.13)]{DHI}). This estimate is not suit for our purpose.
We will obtain shaper result given by the following lemma.

\begin{lemma}\label{lem.modP}
{\it Let $f\in L_p[-1,1]$, $0<p<\infty$, $r\in\N$, $\lambda>0$, and $\vp(x)=\sqrt{1-x^2}$. Then
\begin{equation}\label{eq.lem1.1}
\w_r^{\vp}(f,\,\lambda\d)_p\leq C(1+\lambda)^{r+2(\frac{1}{p_1}-1)}\w_r^{\vp}(f,\d)_p\,,
\end{equation}
where  $C$ is a constant independent of $f$, $\lambda$, and $\d$.
}
\end{lemma}

This lemma, Lemma~\ref{lemJP}, and the inverse
inequality for  the best polynomial approximation (see~\cite{DJL}) allow
us to prove the next theorem about estimates from below for the
best polynomial approximation of functions in $L_p[-1,1]$,
$0<p<1$. To prove this theorem one can use the scheme of
proving the corresponding result in~\cite{Ko07}, see also the
proof of Theorem~\ref{thbelowS}. Note  that in the case $p\ge 1$,
the next theorem was obtained by Rathore~\cite{Rath94}.

\begin{theorem}\label{thCTP}
Let $f\in L_p[-1,1]$, $0<p<1$, and $s,n\in \N$. Then the following
assertions are equivalent:

\smallskip

\noindent $(i)$ for some $k>s+2/p-2$ there exist constants $M>0$ and
$\lambda>0$ such that for any $h\in (\lambda/n,1)$
$$
\w_s^\vp(f,h)_p\le M\w_k^\vp(f,h)_p,
$$
\noindent $(ii)$  there exists a constant $L>0$ independent of $f$
and $n$ such that
$$
\w_s^\vp\(f,\frac1n\)_p\le L E_n(f)_p.
$$
\end{theorem}

To prove Lemma~\ref{lem.modP}, we need the following technical result.

\begin{lemma}\label{lem2AAA}
{\it
Let $r,n\in \N$ and let the numbers $A_{\nu,n}^{(r)}$, $0\leq \nu\leq r(n-1)$, satisfy
\begin{equation*}
    \(1+t+t^{2}+\ldots+t^{n-1}\)^{r}=\sum_{\nu=0}^{(n-1)r}A_{\nu,n}^{(r)}\,t^{\nu}\,,\quad t\in\R.
\end{equation*}
Then for any $0\leq\nu\leq (n-1)r$ we have
\begin{equation}\label{eq.lem2.1}
   A_{\nu,n}^{(r)}=A_{r(n-1)-\nu,n}^{(r)}\,,
\end{equation}
\begin{equation}\label{eq.lem2.2}
   0<A_{\nu,n}^{(r)}\leq C(\nu+1)^{r-1}\,,
\end{equation}
where the constant $C$ depends only on $k$.
}
\end{lemma}

\begin{proof}
Equality (\ref{eq.lem2.1}) can be verified by using the substitute $t\rightarrow {1}/{t}$. The proof of~(\ref{eq.lem2.2}) can be found in~\cite[p.~187]{PePo}.

%
%

\end{proof}

\begin{proof}\emph{of\,\, Lemma~\ref{lem.modP}\,\,}
In the case $1\leq p<\infty$, the proof of the lemma see in \cite[p.~38]{DiTo}.
Let us consider the case $0<p<1$. In what follows, we use some ideas from~\cite{DHI}.

First let us prove
\begin{equation}\label{eq.lem1.2}
\begin{split}
&\int_0^{\sqrt{n}t}\int_{-1}^1
|\bar{\D}_{u\vp(x)}^rf(x)|^p {\rm d}x{\rm d}u\\
&\leq\frac{2n^2}{n+1}\sum_{\nu=0}^{(n-1)r}\(A_{\nu,n}^{(r)}\)^p\frac{1}{\a_{r(n-1)-\nu}}
\int_0^{\a_\nu t/\sqrt{n}}\int_{-1}^1
|\bar{\D}_{u\vp(x)}^rf(x)|^p {\rm d}x{\rm d}u\,,
\end{split}
\end{equation}
where $A_{\nu,n}^{(r)}$, $0\leq \nu \leq r(n-1)$, are defined in
Lemma~\ref{lem2AAA} and
$$\a_\nu=\a_{\nu,n}^{(r)}=\(\frac{rn}{2\(r(n-1)-\nu\)+r}\)^{1/2}\,,
\quad 0 \leq \nu \leq r(n-1)\,.$$
Denote
$$
D_n(r,w)=\left\{x:x\pm\frac{rw n\vp (x)}{2}\in(-1,\,1)\right\}\,.
$$
By using the following equality (see \cite[p.~187]{PePo})
$$
\bar{\D}_{nw\vp(x)}^rf(x)=\sum_{\nu=0}^{(n-1)r}A_{\nu,n}^{(r)}\,\bar{\Delta}_{w\vp(x)}^rf\(x+\(\nu-\frac{r(n-1)}{2}\)w\vp(x)\)\,, \quad x\in D_n(r,w)\,,
$$
we obtain
\begin{equation}\label{eq.lem1.3}
\begin{split}
&\int_0^{\sqrt{n}t}\int_{-1}^1
|\bar{\D}_{u\vp(x)}^rf(x)|^p {\rm d}x{\rm d}u
=n\int_0^{t/\sqrt{n}}\int_{D_n(r,w)}
|\bar{\D}_{nw\vp(x)}^rf(x)|^p {\rm d}x{\rm d}w\\
&\leq n\sum_{\nu=0}^{(n-1)r}\(A_{\nu,n}^{(r)}\)^p\int_0^{t/\sqrt{n}}\int_{D_n(r,w)}
\left|\bar{\D}_{w\vp(x)}^rf\(x+\(\frac{r(n-1)}{2}-\nu\)w\vp(x)\)\right|^p {\rm d}x{\rm d}w\\
&=n\sum_{\nu=0}^{(n-1)r}\(A_{\nu,n}^{(r)}\)^pI_{\nu,n}\,.
\end{split}
\end{equation}

Let us estimate $I_{\nu,n}$, $0\leq\nu\leq(n-1)r$. We set
$$y=x+\(\frac{r(n-1)}{2}-\nu\)w\vp(x)\,.$$
With no loss of generality, we can assume that
\begin{equation}\label{eq.lem1.4}
\frac{r(n-1)}{2}-\nu \leq 0
\end{equation}
(the case $\frac{r(n-1)}{2}-\nu\geq 0$ is symmetric).
In view of $x\pm{rw n\vp (x)}/{2}\in(-1,\,1)$ and (\ref{eq.lem1.4}), we have
$w \vp(x)\leq {2}(1+x)/({rn})$ and $1-y\geq1-x$. Thus,
\begin{equation}\label{eq.lem1.5}
\begin{split}
1-y^2=(1+y)(1-y)&\geq\(1+x+\(\frac{r(n-1)}{2}-\nu\)\frac{2}{rn}(1+x)\)(1-x)\\
&=\frac{2\(r(n-1)-\nu\)+r}{rn}(1-x^2)\,.
\end{split}
\end{equation}
By analogy, using the inequalities $w \vp(x)\leq{2}(1-x)/({rn})$ and $1+y\leq1+x$,
we have
\begin{equation}\label{eq.lem1.6}
1-y^2\leq\(\frac{r+2\nu}{rn}\)(1-x^2)\,.
\end{equation}
Now, denoting $\vp_1(y)=\vp(x)$,  we get from (\ref{eq.lem1.5}) and (\ref{eq.lem1.6}) that
\begin{equation}\label{eq.lem1.7}
\a_{r(n-1)-\nu}\vp(y)\leq\vp_1(y)\leq\a_\nu \vp(y)
\end{equation}
and
\begin{equation}\label{eq.lem1.8}
I_{\nu,n}=\int_0^{t/\sqrt{n}}\int_{S_n(w,r,\nu)}
|\bar{\D}_{w\vp_1(y)}^rf(y)|^p \frac{{\rm d}x}{{\rm d}y}{\rm d}y{\rm d}w\,,
\end{equation}
where
$$
S_n(w,r,\nu)=\left\{y:y=x+\(\frac{r(n-1)}{2}-\nu\)w\vp(x),\,x\pm\frac{r\nu n}{2}\vp(x)\in(-1,1)\right\}.
$$

It is easy to see that
$$
\frac{{\rm d}y}{{\rm d}x}=1-\(\frac{r(n-1)}{2}-\nu\)w\frac{x}{\vp(x)}\,.
$$
Tacking into account (\ref{eq.lem1.4}), we have for $x\geq0$ that ${{\rm d}y}/{{\rm d}x}\geq1$. For $x<0$, using inequality $w \vp(x)\leq {2}(1+x)/({rn})$, we get
\begin{equation*}
\begin{split}
\frac{{\rm d}y}{{\rm d}x}&\geq1+\frac{r(n-1)}{2}w\frac{x}{\vp(x)}=1+\frac{r(n-1)}{2}w\vp(x)\frac{x}{1-x^2}\\
&\geq1+\frac{(n-1)(1+x)x}{n(1-x^2)}=1+\frac{n-1}{n}\cdot\frac{x}{1-x}\geq\frac{n+1}{2n}\,.
\end{split}
\end{equation*}
Hence, for $y\in S_n(w,r,\nu)$, we have ${{\rm d}x}/{{\rm d}y}\leq{2n}/({n+1})$ and
\begin{equation}\label{eq.lem1.9}
I_{\nu,n}\leq\frac{2n}{n+1}\int_0^{t/\sqrt{n}}\int_{-1}^1
|\bar{\Delta}_{w\vp_1(y)}^rf(y)|^p {\rm d}y{\rm d}w\,.
\end{equation}
In \cite{DHI}, it was proved that for $0<B^{-1}\psi(x)\leq\psi_1(x)\leq A\psi(x)$ and $0<p<\infty$ one has
\begin{equation}\label{eq.lem1.10}
\int_0^{t}\int_{-1}^1
|\bar{\D}_{w\psi_1(x)}^rf(x)|^p {\rm d}x{\rm d}w\leq B\int_0^{At}\int_{-1}^1
|\bar{\D}_{u\psi(x)}^rf(x)|^p {\rm d}x{\rm d}u\,.
\end{equation}
Using (\ref{eq.lem1.7}), (\ref{eq.lem1.10}), and (\ref{eq.lem1.9}), we obtain
\begin{equation}\label{eq.lem1.11}
I_{\nu,n}\leq\frac{2n}{(n+1)\a_{r(n-1)-\nu}}\int_0^{{\a_\nu t}/{\sqrt{n}}}\int_{-1}^1
|\bar{\Delta}_{u\vp(x)}^rf(x)|^p {\rm d}x{\rm d}u\,.
\end{equation}
Thus, combining (\ref{eq.lem1.3}) and (\ref{eq.lem1.11}), we get (\ref{eq.lem1.2}).

Next, using  (\ref{eq.lem1.2}) and  the
following two-sided inequality with positive constants independent of $f$ and $t$ (see \cite[Corollary 5.5]{DHI})
$$
\w_r^\vp(f,t)_p^p\asymp\frac{1}{t}\int_0^{t}\int_{-1}^1
|\bar{\D}_{w\vp(x)}^rf(x)|^p {\rm d}x{\rm d}w\,,
$$
we derive
\begin{equation}\label{eq.lem1.12}
\begin{split}
\w_r^\vp(f,\sqrt{n}t)_p^p&\leq\frac{C}{\sqrt{n}t}\int_0^{\sqrt{n}t}\int_{-1}^1
|\bar{\D}_{w\vp(x)}^rf(x)|^p {\rm d}x{\rm d}w\\
&\leq C\sqrt{n}\sum_{\nu=0}^{r(n-1)}\(A_{\nu,n}^{(r)}\)^p\frac{1}{\a_{r(n-1)-\nu}t}
\int_0^{\a_\nu t/\sqrt{n}}\int_{-1}^1
|\bar{\D}_{w\vp(x)}^rf(x)|^p {\rm d}x{\rm d}w\\
&\leq C\sum_{\nu=0}^{r(n-1)}{\(A_{\nu,n}^{(r)}\)^p}\frac{\a_\nu}{\a_{r(n-1)-\nu}}\w_r^\vp\(f,\frac{\a_\nu}{\sqrt{n}}t\)_p^p\,.
\end{split}
\end{equation}

Let $P_N\in\mathcal{P}_{N-1}$, $N\in\N$, be such that $\Vert f-P_N\Vert_p=E_N(f)_p$. Let us choose $N\in\N$ such that ${t}/{2}<N^{-1}\leq t$. By (\ref{eq.lem1.12}), Lemmas~\ref{lemJP} and~\ref{lemNSP}, inequalities \eqref{eq.Sec3.M1} and (\ref{eq.*}), and
$$
\frac{\a_\nu}{\sqrt{n}}=\(\frac{r}{2\(r(n-1)-\nu\)+r}\)^{1/2}\le 1\,,
$$
we obtain
\begin{equation}\label{eq.lem1.13}
\begin{split}
\w_r^\vp(f,\sqrt{n}t)_p^p&\leq C \Vert f-P_N\Vert_p^p+\w_r^\vp(P_N,\sqrt{n}t)_p^p\\
&\leq C\w_r^\vp\(f,\frac{1}{N}\)_p^p+C\sum_{\nu=0}^{r(n-1)}\(A_{\nu,n}^{(r)}\)^p\frac{\a_\nu}{\a_{r(n-1)-\nu}}\w_r^\vp\(P_N,\frac{\a_\nu}{\sqrt{n}}t\)_p^p\\
&\leq C\w_r^\vp(f,t)_p^p+{C}{n^{-\frac{rp}{2}}}\sum_{\nu=0}^{r(n-1)}\(A_{\nu,n}^{(r)}\)^p\frac{\a_\nu^{1+rp} t^{rp}}{\a_{r(n-1)-\nu}}\Vert\vp^r P_N^{(r)}\Vert_p^p\\
&\leq C\w_r^\vp(f,t)_p^p+{C}{n^{-\frac{rp}{2}}}\sum_{\nu=0}^{r(n-1)}{\(A_{\nu,n}^{(r)}\)^p}
\frac{\a_\nu^{1+rp}}{\a_{r(n-1)-\nu}}\w_r^\vp\(P_N,\frac{1}{N}\)_p^p\\
&\leq C\w_r^\vp(f,t)_p^p\\
&\quad\quad+{C}{n^{-\frac{rp}{2}}}\sum_{\nu=0}^{r(n-1)}{\(A_{\nu,n}^{(r)}\)^p}
\frac{\a_\nu^{1+rp}}{\a_{r(n-1)-\nu}}\(\Vert f-P_N\Vert_p^p+\w_r^\vp\(f,\frac{1}{N}\)_p^p\)\\
&\leq C\( 1+{n^{-\frac{rp}{2}}}\sum_{\nu=0}^{r(n-1)}{\(A_{\nu,n}^{(r)}\)^p}\frac{\a_\nu^{1+rp}}{\a_{r(n-1)-\nu}}\)\w_r^\vp(f,t)_p^p\,.
\end{split}
\end{equation}
Using inequalities (\ref{eq.lem2.1}) and (\ref{eq.lem2.2}), it is easy to verify that
\begin{equation}\label{eq.lem1.14}
\begin{split}
\sum_{\nu=0}^{r(n-1)}\(A_{\nu,n}^{(r)}\)^p\frac{\a_\nu^{1+rp}}{\a_{r(n-1)-\nu}}&\leq Cn^{1+p(r-1)}\,.
\end{split}
\end{equation}
Thus, by (\ref{eq.lem1.13}) and (\ref{eq.lem1.14}), we have
$
\w_r^\vp(f,\sqrt{n}t)_p\leq Cn^{{1}/{p}-1+{r}/{2}}\w_r^\vp(f,t)_p\,,
$
which implies that
$
\w_r^\vp(f,{n}t)_p\leq Cn^{2({1}/{p}-1)+r}\w_r^\vp(f,t)_p
$
and, therefore, (\ref{eq.lem1.1}) is true.

The lemma is proved.
\end{proof}

Now we are ready to prove Theorem~\ref{th3}. Actually, this theorem can be
proved by using Lemmas~\ref{lemJP}--\ref{lem.modP} and repeating the proof of Theorem~\ref{thsec3.1}. Nevertheless, since we have some changes in the proof,  we would like to present a detailed proof of the theorem.


\begin{proof}\emph{of\,\, Theorem~\ref{th3}\,\,\,}
Let $P_{n}\in \mathcal{P}_{n}$,  $n\in \N$, be such that
\begin{equation}\label{++++}
  \Vert \vp^r (f^{(r)}-P_{n}^{(r)})\Vert_p=E_{n-r}(f^{(r)})_{p,\vp^r}.
\end{equation}
Let us choose $m\in \N$  such that
$2^{-(m+1)}\le n< 2^{m}$.
We have
\begin{equation}\label{eq.th3.2}
\begin{split}
 E_n(f)_p^p&\le E_n(P_{2^m})_p^p+E_n(f-P_{2^m})_p^p.
\end{split}
\end{equation}
By Lemmas~\ref{lemJP} and~\ref{lemNSP}, we obtain
\begin{equation}\label{eq.th3.3}
\begin{split}
E_n(P_{2^m})_p&=E_n(P_{2^m}-P_n)_p\le C\w_r^\vp(P_{2^m}-P_n,n^{-1})_p\\
&\le Cn^{-r}\Vert \vp^r(P_{2^m}^{(r)}-P_n^{(r)})\Vert_p\le Cn^{-r}E_n(f^{(r)})_{p,\vp^r}.
\end{split}
\end{equation}

Let us show that
\begin{equation}\label{eq.th3.6}
\begin{split}
E_n(f-P_{2^m})_p^p\le \sum_{\nu=m}^\infty
E_n(P_{2^{\nu+1}}-P_{2^\nu})_p^p.
\end{split}
\end{equation}
Indeed, for any $N>m$ we have
\begin{equation*}
\begin{split}
E_n(f-P_{2^m})_p^p\le \sum_{\nu=m}^{N-1}
E_n(P_{2^{\nu+1}}-P_{2^\nu})_p^p+E_n(f-P_{2^{N}})_p^p\,.
\end{split}
\end{equation*}
Thus, to prove (\ref{eq.th3.6}), one needs only to verify that
\begin{equation}\label{eq.th3.8}
\begin{split}
E_n(f-P_{2^{N}})_p\to 0\quad\text{as}\quad N\to
\infty.
\end{split}
\end{equation}
Using Lemma~\ref{lemJP}, H\"older's inequality, and
the following estimate (see, e.g., \cite[Theorem~2.1.1]{DiTo})
\begin{equation*}\label{eq.th3.9}
\begin{split}
\w_r^\vp (f,\d)_1\le C_r \d^r \Vert \vp^r f^{(r)}\Vert_1,\quad 0<\d\le
\d_0,
\end{split}
\end{equation*}
we get
\begin{equation}\label{eq.th3.10}
\begin{split}
E_n(f-P_{2^{N}})_p&\le C\w_r^\vp (f-P_{2^{N}},n^{-1})_p\le C\w_r^\vp (f-P_{2^{N}},n^{-1})_1\\
&\le Cn^{-r}\Vert \vp^r
(f-P_{2^{N}})^{(r)}\Vert_1.
\end{split}
\end{equation}
By Lemma~\ref{lemNikP}, we obtain
\begin{equation*}
\begin{split}
\sum_{\nu=m}^\infty \Vert \vp^r (P_{2^{\nu+1}}^{(r)}-P_{2^{\nu}}^{(r)})\Vert_{1}^p
&\le C\sum_{\nu=m}^\infty 2^{2(1-p)\nu}\Vert \vp^r (P_{2^{\nu+1}}^{(r)}-P_{2^{\nu}}^{(r)})\Vert_{p}^p\\
&\le C\sum_{\nu=m}^\infty
2^{2(1-p)\nu}E_{2^\nu-r}(f^{(r)})_{p,\vp^r}^p.
\end{split}
\end{equation*}
In view of~\eqref{eq.th3.0}, the last inequality implies that  $\{\vp^r P_{2^\nu}^{(r)}\}$ is
convergent in $L_1[-1,1]$. At the same time, by~\eqref{++++}, we have $\vp^r P_{2^\nu}^{(r)}\to \vp^r
f^{(r)}$ as $\nu\to\infty$ in $L_p[-1,1]$. Thus, $\vp^r
P_{2^\nu}^{(r)}\to \vp^r f$ as $\nu\to\infty$  in $L_1[-1,1]$, too.
From this and (\ref{eq.th3.10}), we obtain (\ref{eq.th3.8}) and,
therefore, \eqref{eq.th3.6}.

Next, using Lemmas~\ref{lemJP},~\ref{lem.modP}, and~\ref{lemNSP}, we derive
\begin{equation}\label{eq.th3.11}
\begin{split}
E_n&(P_{2^{\nu+1}}-P_{2^\nu})_p\le C\w_{r}^\vp(P_{2^{\nu+1}}-P_{2^\nu},n^{-1})_p\\
&=C \w_{r}^\vp(P_{2^{\nu+1}}-P_{2^\nu},M2^{\nu+1}n^{-1}(M2^{\nu+1})^{-1})_p\\
&\le C n^{-r-2/p+2}2^{\nu(r+2/p-2)}\w_{r}^\vp(P_{2^{\nu+1}}-P_{2^\nu},(M2^{\nu+1})^{-1})_p\\
&\le Cn^{-r-2/p+2}  2^{(2/p-2)\nu}\Vert \vp^r (P_{2^{\nu+1}}-P_{2^\nu})^{(r)}\Vert_{p}\\
&\le Cn^{-r-2/p+2}
2^{(2/p-2)\nu}E_{2^\nu-r}(f^{(r)})_{p,\vp^r}.
\end{split}
\end{equation}

Finally, combining (\ref{eq.th3.2}), (\ref{eq.th3.3}), (\ref{eq.th3.6}), and (\ref{eq.th3.11}), we obtain
(\ref{eq.th3.1}).

The theorem is proved.

\end{proof}

Now let us consider analogues of Theorems~\ref{thsec3.A}
and~\ref{thsec3.2} in the case of the approximation of functions by
algebraic polynomials.
Following~\cite{DiTi07}, we  say that $f$ has (weak) $r$th
derivative in $L_p[-1,1]$ if there exists a function $g$ such that
for any interval $[a,b]$, $-1<a<b<1$, one has
\begin{equation}\label{SenseSerL[-1,1]}
  \left\Vert\frac{\bar{\D}^r f}{h^r}-g\right\Vert_{L_p[a,b]}\to 0 \quad\text{as}\quad h\to 0+.
\end{equation}
In this case we write $g=f^{(r)}$.

\begin{theorem}\label{thsec3.AP}
{\it Let $f\in L_p[-1,1]$, $0<p<1$, and let for some $k\in \N$
\begin{equation*}
\sum\limits_{\nu=1}^\infty \nu^{kp-1}E_\nu(f)_p^p<\infty\,.
\end{equation*}
Then $f$ has the derivative $f^{(k)}$ in the sense (\ref{SenseSerL[-1,1]}) and for any $n>k$ one has
\begin{equation*}
\|\vp^k(f^{(k)}-P_n^{(k)})\|_p\leq
C\(n^kE_n(f)_p+\(\sum\limits_{\nu=n+1}^\infty
\nu^{kp-1}E_\nu(f)_p^p\)^{\frac{1}{p}}\)\,,
\end{equation*}
where $P_n\in\mathcal{P}_n$,
$n\in\N$, are such that $\|f-P_n\|_p=E_n(f)$ and $C$ is a constant independent of $f$ and $n$. }
\end{theorem}

\begin{proof}
This theorem can be proved by using the scheme of the proof of
Theorem~\ref{th2S}. See also the proof of Theorem~5.13
in~\cite{DiTi07}. An analogue of Lemma~\ref{lem.th2} can be also
find in~\cite[Corollary 4.12]{DiTi07}.
\end{proof}

The corresponding theorem about simultaneous approximation by
algebraic polynomials has the following form.

\begin{theorem}\label{thsec3.2P}
{\it Let $0<p<1$, $r\in \N$, and let $f$ be such that $f^{(r-1)}\in
AC[-1,\,1]$, $\vp^rf^{(r)}\in L_1[-1,1]$, and
\begin{equation}\label{eqthsec3.2.1P}
  \sum\limits_{\nu=1}^\infty \nu^{1-2p}E_{\nu-r}(f^{(r)})_{p,\vp^r}^p<\infty\,.
\end{equation}
Then for any $n\ge n(r)$ we have
\begin{equation*}
\begin{split}
  &\|\vp^r(f^{(r)}-P_n^{(r)})\|_p\leq C\Bigg(E_{n-r}(f^{(r)})_{p,\vp^r}+n^{2-\frac{2}{p}}\bigg(\sum\limits_{\nu=n+1}^\infty \nu^{1-2p}E_{\nu-r}(f^{(r)})_{p,\vp^r}^p\bigg)^{\frac{1}{p}}\Bigg)\,,
  \end{split}
\end{equation*}
where $P_n\in\mathcal{P}_n$, $n\in\N$, are such that $\|f-P_n\|_p=E_n(f)$ and $C$ is a constant independent of $f$ and $n$.
}
\end{theorem}

\begin{proof}
To prove the theorem one can use Theorem~\ref{th3} and Theorem~\ref{thsec3.AP} and repeat the scheme of proving Theorem~\ref{thsec3.2}.
\end{proof}

From Theorems~\ref{th3},~\ref{thsec3.AP}, and~\ref{thsec3.2P}, we get the following corollary.

\begin{corollary}\label{corsec3.1E_Alg}
{\it Let $0<p<1$, $r\in \N$, $\a>{2}/{p}-2$, and let $f$ be such
that $f^{(r-1)}\in AC[-1,\,1]$, $\vp^rf^{(r)}\in L_1[-1,1]$. Then
the following assertions are equivalent:

$(i)$ $E_n(f)_p=\mathcal{O}(n^{-r-\a})\,, \quad n\rightarrow
\infty\,,$

$(ii)$ $E_n(f^{(r)})_{p,\vp^r}=\mathcal{O}(n^{-\a})\,, \quad
n\rightarrow \infty\,,$

$(ii)$ $\Vert
\vp^r(f^{(r)}-P_n^{(r)})\Vert_p=\mathcal{O}(n^{-\a})\,, \quad
n\rightarrow \infty\,,$

\noindent where $P_n\in\mathcal{P}_n$, $n\in\N$, are such that
$\|f-P_n\|_p=E_n(f)$.}
\end{corollary}

\subsection{\textbf{Inequalities for the Ditzian-Totik moduli of smoothness}}

By Theorems~\ref{th3},~\ref{thsec3.AP}, and~\ref{thsec3.2P}, we can prove the next two theorems for
the Ditzian-Totik moduli of smoothness of functions and their
derivatives.

\begin{theorem}\label{th3m}
{\it Let $0<p<1$, $r,\,k\in\N$, and let a function $f$ be such that $f^{(r-1)}\in AC[-1,\,1]$, $\vp^rf^{(r)}\in L_1[-1,1]$. Then for any $\d\in(0,\,\d_0)$ we have
\begin{equation}\label{eq.th3.1m+}
\begin{split}
\w_{r+k}^\vp(f,\,\d)_p\leq
&C\d^r\Omega_k^\vp(f^{(r)},\,\d)_{p,\,\vp^r}\\
&+C\d^{r+\frac{2}{p}-2}\(\sum_{\nu=\left[{1}/{\d}\right]}^\infty\nu^{1-2p}E_{\nu-r}(f^{(r)})_{p,\vp^r}^p\)^{\frac{1}{p}}\,,
\end{split}
\end{equation}
where $\d_0$ and $C$ are constants independent of $f$ and $\d$.
}
\end{theorem}

\begin{proof}\emph{of\,\, Theorem~\ref{th3m}\,\,\,}
It is clear that we can suppose that~\eqref{eqthsec3.2.1P} holds.

Let $P_{n}\in \mathcal{P}_{n}$,  $n\in \N$, be such that
$$
\Vert f-P_{n}\Vert_p=E_{n}(f)_{p}.
$$
Let us choose $n$  such that
$2^{-n}\le \d< 2^{-(n+1)}$.
We have
\begin{equation}\label{eq.th3.2+}
\begin{split}
 \w_{r+k}^\vp (f,\d)_p^p\le \w_{r+k}^\vp(f-P_{2^n},\d)_p^p+\w_{r+k}^\vp(P_{2^n},\d)_p^p.
\end{split}
\end{equation}
By Lemma~\ref{lemNSP}, we get
\begin{equation}\label{eq.th3.3+}
\begin{split}
\w_{r+k}^\vp(P_{2^n},\d)_p^p &\le C\Omega_{r+k}^\vp (P_{2^n}, \d)_p^p\le C\d^{rp}\Omega_{k}^\vp (P_{2^n}^{(r)}, \d)_{p,\vp^r}^p\\
&\le C\d^{rp}(\Omega_{k}^\vp (f^{(r)},\d)_{p,\vp^r}^p+\Omega_{k}^\vp (f^{(r)}-P_{2^n}^{(r)},\d)_{p,\vp^r}^p).
\end{split}
\end{equation}
Note that if $g\in L_p[a,b]$ for any $-1<a<b<1$, then
\begin{equation*}\label{eq.th3.4+}
\Omega_k^\vp (g,\d)_{p,\vp^r}\le C\Vert \vp^r
g\Vert_{p}
\end{equation*}
(see, e.g., the proof of Lemma~5.15 in~\cite{DiTi07}). Therefore,
\begin{equation}\label{eq.th3.5+}
\begin{split}
\Omega_k^\vp (f^{(r)}-P_{2^n}^{(r)},\d)_{p,\vp^r}^p&\le C\Vert
\vp^r(f^{(r)}-P_{2^n}^{(r)})\Vert_{p}^p.
\end{split}
\end{equation}

Now, combining~\eqref{eq.th3.2+}--\eqref{eq.th3.5+} and using~\eqref{eq.Sec3.M1}, we derive
\begin{equation}\label{eq.th3.6+}
  \begin{split}
    \w_{r+k}^\vp (f,\d)_p^p\le C(\d^{rp}\Omega_{k}^\vp (f^{(r)},\d)_{p,\vp^r}^p&+\d^{rp}\Vert
\vp^r(f^{(r)}-P_{2^n}^{(r)})\Vert_{p}^p\\
&+\Vert f-P_{2^n}\Vert_{p}^p).
  \end{split}
\end{equation}
Next, by Theorem~\ref{th3} (more precisely, using \eqref{eq.th3.6} and~\eqref{eq.th3.11}), we derive
\begin{equation}\label{eq.th3.7+}
  \begin{split}
\Vert f-P_{2^n}\Vert_{p}^p\le C2^{-n(rp+2-2p)}\sum_{\nu=n}^\infty
2^{2(1-p)\nu}E_{2^\nu-r}(f^{(r)})_{p,\vp^r}^p.
  \end{split}
\end{equation}
Similarly, by Theorem~\ref{thsec3.2P}, we have
\begin{equation}\label{eq.th3.8+}
  \begin{split}
\Vert
\vp^r(f^{(r)}-P_{2^n}^{(r)})\Vert_{p}^p\le C2^{-n(2-2p)}\sum_{\nu=n}^\infty
2^{2(1-p)\nu}E_{2^\nu-r}(f^{(r)})_{p,\vp^r}^p.
  \end{split}
\end{equation}

Finally, combining (\ref{eq.th3.6+}), (\ref{eq.th3.7+}), and (\ref{eq.th3.8+}), we get
(\ref{eq.th3.1m+}).

The theorem is proved.

\end{proof}

Recall also the following counterpart of Theorem~\ref{thsec3.B} and Theorem~\ref{thModFrD} in the case of the Ditzian-Totik moduli of smoothness (see~\cite{DiTi07}).

\begin{theorem}\label{thsec3.BP}
{\it Let $f\in L_p[-1,1]$, $0<p<1$, $r\,,k\in \N$, and $k<r$. Then
\begin{equation}\label{eqthsec3.B.1P}
    \Omega_{r-k}^\vp(f^{(k)},\d)_{p,\vp^k}\le C\(\int_0^\d
    \frac{\w_r^\vp(f,t)_p^p}{t^{pk+1}}{\rm d}t\)^\frac1p, \quad \d>0,
\end{equation}
where $C$ is some constant independent of $f$ and
$\d$. Inequality (\ref{eqthsec3.B.1P}) means that if the right-hand side is finite, then there exists
 $f^{(k)}$ in the sense
(\ref{SenseSerL[-1,1]}), $\vp^k f^{(k)}\in L_p[-1,1]$, and
(\ref{eqthsec3.B.1P}) holds.}
\end{theorem}

Finally, from Theorem~\ref{th3m},  Theorem~\ref{thsec3.BP},
Corollary~\ref{corsec3.1E_Alg}, and Lemma~\ref{lemJP}, we get the
following result.
\begin{corollary}\label{corsec3.1ModP}
{\it Under the conditions of Corollary~\ref{corsec3.1E_Alg}, the
following assertions are equivalent for any $k\in \N$:

$(i)$ $\w_{r+k}^\vp(f,\d)_p=\mathcal{O}(\d^{r+\a})\,, \quad
\d\rightarrow 0\,,$

$(ii)$ $\Omega_{k}^\vp(f^{(r)},\d)_{p,\vp^r}=\mathcal{O}(\d^{\a})$
and $E_n(f^{(r)})_{p,\vp^r}=\mathcal{O}(n^{-\a})$ as $\d\rightarrow
0$ and $n\to \infty$.}
\end{corollary}




\end{document}